\documentclass[12pt]{amsart}
\usepackage{amssymb}
\usepackage{latexsym}
\usepackage{amsmath}

\def\wt{\widetilde}
\def\wh{\widehat}
\def\ov{\overline}
\def \im{{\rm Im\,}}
 \def\up{\upharpoonright}

\def\cA{\mathcal A} \def\cH{\mathcal H}  \def\cB{\mathcal B} 
  \def\cF {\mathcal F}

\def\cK{\mathcal K} \def\cL{\mathcal L}
 \def\cN{\mathcal N}

\def\B{\mbox{\boldmath$B$}}

\def \gH{\mathfrak H}   \def \gN{\mathfrak N}

\def \bC{\mathbb C}  
\def\bR{\mathbb R}
 
\def \l{\lambda}
\def \a{\alpha}       \def \t{\theta}
 \def\g {\gamma}

\def \f{\varphi}  
 \def \G{\Gamma} 
\def \C{\widetilde {\mathcal C}}
\def \CA {C(\wt A)} \def \CAt {C(\wt A_\tau)}

\def \cd {\cdot}

\def \ex { {\rm ext} (A)} \def \cex {\overline {\rm ext }(A)}
\def\CAt{C(\wt A_\tau)} \def\TAt{T(\wt A_\tau)} \def\SAt{S(\wt A_\tau)}

\def \dom {{\rm dom}\,}  \def \ran {{\rm ran}\,}  \def \ker{{\rm
ker\,}}
 \def \mul {{\rm mul}\,}

\def  \RH {\wt R (\cH)}  \def  \RCH {\wt R_c (\cH)} \def\Se{{\rm Self} (A)} \def\Sez{{\rm Self}_0 (A)}

\def \CR {\bC\setminus\bR}

\def\bt{\{\cH,\G_0,\G_1\}}

\newtheorem{theorem}{Theorem}[section]
\newtheorem{proposition}[theorem]{Proposition}
\newtheorem{corollary}[theorem]{Corollary}
\newtheorem{lemma}[theorem]{Lemma}

\theoremstyle{definition}

\theoremstyle{definition}
\newtheorem {definition} [theorem]{Definition}
\theoremstyle{remark}
\newtheorem{remark}[theorem]{Remark}
\numberwithin{equation}{section}
\begin{document}
\title[On compressions of self-adjoint extensions]
{On compressions of self-adjoint extensions of a symmetric linear relation  }
\author{V.I. Mogilevskii}
\address{Department of Mathematical Analysis and Informatics, Poltava National V.G. Korolenko Pedagogical University,  Ostrogradski Str. 2, 36000 Poltava, Ukraine }
\email{vadim.mogilevskii@gmail.com}

%\date{DD/MM/2004}
\begin{abstract}
Let $A$ be a symmetric linear relation in the Hilbert space $\gH$ with equal deficiency indices $n_\pm (A)\leq\infty$. A self-adjoint linear relation $\wt A\supset A$ in some Hilbert space $\wt\gH\supset \gH$ is called an exit space extension of $A$; such an extension is called finite-codimensional if $\dim (\wt\gH\ominus\gH)< \infty$. We study the compressions $C (\wt A)=P_\gH\wt A\up\gH$ of exit space extensions $\wt A=\wt A^*$. For a certain class of extensions $\wt A$ we parameterize  the compressions $C (\wt A)$  by means of abstract boundary conditions. This enables  us to characterize various properties of $C (\wt A)$ (in particular, self-adjointness) in terms of the parameter for $\wt A$ in the Krein formula for resolvents. We describe also the compressions of a certain class of finite-codimensional extensions. These results develop the  results by A. Dijksma and H. Langer obtained for a densely defined symmetric operator $A$ with finite deficiency indices.
\end{abstract}
\maketitle
\section{Introduction}
Let $\gH$ be a subspace in a Hilbert space  $\wt\gH$  and let $\wt A$ be a linear relation (in particular operator) in $\wt \gH$. Recall that a linear relation $C(\wt A)$ in $\gH$ given by $C(\wt A)=P_\gH\wt A\up\gH$ is called a compression of $\wt A$; moreover, the relation $\wt A$ and its compression $C(\wt A)$ are called finite-codimensional if $\dim (\wt\gH\ominus\gH)<\infty$. Compressions of linear operators or relations were recently studied in \cite{AziDaj12,ADW13,ACD16,DajLan17,DajLan18,Nud11}. In particular, it was shown in \cite{ADW13} that a finite-codimensional compression $C(\wt A)$ of a self-adjoint linear relation $\wt A$ is self-adjoint (for operators $\wt A=\wt A^*$ this fact was established earlier in  \cite{Ste68}).

Assume now that $A$ is a not necessarily densely defined symmetric operator in a Hilbert space $\gH$. A self-adjoint linear relation $\wt A\supset A$ in a Hilbert space $\wt\gH\supset \gH$ is called an exit space extension of $A$. Denote by $\Se$ the set of all exit space extensions $\wt A=\wt A^*$  of $A$ and by $\Sez$ the set of all $\wt A\in\Se$ such that $\wt A$ is an operator. As is known $\Se=\Sez$ if and only if $A$ is densely defined.

If $\wt A\in\Se$, then the compression $C(\wt A)$ of $\wt A $ is a symmetric extension of $A$. A description of all extensions $\wt A\in\Se$ and their compressions $C (\wt A)$ is an important problem in the extension theory of symmetric operators. In the paper by A.V. Shtraus \cite{Sht70} all extensions $\wt A\in\Sez$ of an operator $A$ with arbitrary (equal or unequal) deficiency indices $n_\pm(A)\leq \infty$  are parameterized by means of holomorphic operator-functions $F(\l), \; \l\in\bC_+,$ whose values are contractions from $\gN_i$ to $\gN_{-i}$ (here  $\gN_{\pm i}=\ker(A^* \mp i)$ are defect subspaces of $A$). In \cite{DajSno74} the result of \cite{Sht70} was extended to extensions $\wt A\in\Se$. In \cite{Sht66,Zag13} the compressions  $C(\wt A)$ of the  extensions  $\wt A\in\Sez$ we described in terms of the asymptotic behaviour of the corresponding parameter $F(\l)$.

In the case $n_+(A)=n_-(A)$ another parametrization  of the set $\Se$ is given by the Krein formula for generalized resolvents \cite{KreLan71,LanTex77}. This formula (see \eqref{2.31}) establishes a bijective correspondence $\wt A =\wt A_\tau$ between all relation-valued Nevanlinna functions $\tau=\tau(\l)\;(\l\in\CR)$ ($\tau (\l)$ is a linear relation in an auxiliary Hilbert space $\cH$) and all extensions $\wt A\in\Se$. The canonical self-adjoint extension $A_0$ of $A$ in \eqref{2.31} is called a basic extension. An operator-valued parameter $\tau=\tau(\l)$ in \eqref{2.31} is called rational if
\begin{gather}\label{1.1}
\tau(\l)=\cA+\l\cB + \sum_{j=1}^l \frac 1 {\a_j-\l}\cA_j, \quad \l\in\CR,
\end{gather}
where $\a_j\in\bR$ and $\cA=\cA^*, \; \cB\geq 0$ and $\cA_j\geq 0$ are bounded operators in $\cH$. In the case $n_+(A)=n_-(A)<\infty$ an extension $\wt A_\tau\in\Se$ is finite-codimensional if and only if a parameter $\tau$ is rational.

In the paper by H. Langer and A. Dijksma \cite{DajLan18} the compressions $C(\wt A_\tau)$ of extensions $\wt A_\tau$ are investigated in terms of the parameter $\tau$ from the Krein formula \eqref{2.31}. The main results of \cite{DajLan18} can be formulated in the form of the following theorem.
\begin{theorem}\label{th1.1}
Assume that $A$ is a densely defined, closed symmetric operator in $\gH$ with finite deficiency indices $n_+(A)=n_-(A)<\infty$ and $A_0=A_0^*$ is the basic extension in the Krein formula \eqref{2.31}. Let for simplicity a parameter $\tau=\tau(\l), \; \l\in\CR,$ in  \eqref{2.31} be an operator-valued function, let $\wt A_\tau\in\Sez$ be the corresponding extension of  $A$ and let $\CAt$ be the compression of $\wt A_\tau$. Assume also that $\cB_\tau\geq 0$ is an operator in $\cH$ given by $\cB_\tau=\lim\limits_{y\to \infty} \frac 1 {iy} \tau(iy) $. Then:

{\rm (i)} If
\begin{gather}\label{1.2}
\lim_{y\to\infty} y \im (\tau(iy)h,h)=\infty, \quad h\in\cH, \;\; h\neq 0,
\end{gather}
then $\CAt\subset A_0$.

{\rm (ii)} If \eqref{1.2} holds, then $\CAt=A$ if and only if $\cB_\tau=0$.

{\rm (iii)} If $\cB_\tau>0$, then $\CAt=A_0$. If \eqref{1.2} holds and $\CAt=A_0$, then $\cB_\tau>0$. Moreover, if $\tau$ is a rational parameter \eqref{1.1}, then $\CAt=A_0\iff \cB_\tau>0$.

{\rm (\romannumeral 4)} If $\tau$ is the rational parameter \eqref{1.1}, then the extension $\wt A_\tau$ is finite-codimensional and the canonical self-adjoint extension $\CAt$ of $A$ corresponds in the Krein formula to the self-adjoint linear relation $\tau_\infty$ in $\cH$ given by
\begin{gather}\label{1.3}
\tau_\infty =\{\{h, P_{\ker \cB}\cA h+h'\}: h\in\ker \cB, h'\in\ran \cB\}.
\end{gather}
\end{theorem}
It is also shown in \cite{DajLan18} that $\tau_\infty$ admits the representation
\begin{gather}\label{1.4}
\tau_\infty=\{\{h,h'\}\in\cH^2: \exists h(\l)\in\cH: h=\lim_{\l\to\infty} h(\l),\; h'=\lim_{\l\to\infty}\tau(\l) h(\l) \}.
\end{gather}

In the present paper we study compressions of extensions  $\wt A\in\Se$ of a symmetric linear relation (in particular, not necessarily  densely defined symmetric  operator) $A$ with possibly infinite deficiency indices $n_+(A)=n_-(A)\leq \infty$. Our approach is based on the theory of boundary triplets  and their  Weyl functions (see \cite{DM91,GorGor,Mal92} and references therein). Recall \cite{GorGor,Mal92} that a collection $\Pi=\bt$ consisting of an  auxiliary Hilbert space $\cH$ and linear mappings $\G_j: A^*\to \cH,\; j\in\{0,1\},$  is called a boundary triplet for $A^*$ if the mapping $(\G_0,\G_1)^\top$ is surjective and the abstract Green identity \eqref{2.10} is valid (here $A^*$ is the adjoint relation). If $\Pi=\bt$ is a boundary triplet for $A^*$, then the equality (the abstract boundary conditions)
\begin {equation}\label{1.5}
\t\to  A_\t :=\{ \hat f\in A^*:\{\G_0  \hat f,\G_1 \hat f \}\in
\t\}
\end{equation}
gives a parametrization of all proper extensions $\wt A$ of $A$ (i.e., all linear relations $\wt A$ in $\gH$ with $A\subset \wt A\subset A^*$)  in terms of linear relations  $\t$ in $\cH$. Moreover, it was shown in \cite{DM91,Mal92} that each boundary triplet $\Pi$ for $A^*$ gives rise to the Krein formula \eqref{2.31} with coefficients $A_0, \g(\l)$ and $M(\l)$ naturally defined in terms of $\Pi$.

Assume that $A$ is a closed symmetric linear relation in $\gH$ with equal deficiency indices $n_+(A)=n_-(A)\leq\infty$ and let $\Pi=\bt$ be a boundary triplet for $A^*$. In the paper we consider extensions $\wt A_\tau\in\Se$ corresponding to a certain subclass of Nevanlinna parameters $\tau(\l)$ in \eqref{2.31}. We show that the compression $\CAt$ of $\wt A_\tau$ being a symmetric extension of $A$ admits the representation $\CAt= A_{\t_c}$ (see \eqref{1.5}) with a symmetric linear relation $\t_c$, which in a certain sense is a limit value of $\tau(iy)$ when $y\to\infty$. This result enables us to characterise various properties of $\CAt$ (self-adjointness,  inclusion $\CAt\subset A_0$, equalities $\CAt=A$ and $\CAt=A_0$) in terms of the parameter  $\tau$. In the case $n_\pm (A)<\infty$ the obtained results become valid for any Nevanlinna parameter $\tau$. This fact enables us to prove the following theorem which strengthens essentially statements (i) -- (iii) of Theorem \ref{th1.1}.
\begin{theorem}\label{th1.2}
Assume that $A$ is a closed symmetric linear relation in $\gH$ with finite deficiency indices $n_+(A)=n_-(A) < \infty$ and let all other assumptions of Theorem \ref{th1.1} be satisfied. Then:

{\rm (i$^\prime$)} $\CAt\subset A_0$ if and only if \eqref{1.2} holds.

{\rm (ii$^\prime$)} $\CAt=A$ if and only $\cB_\tau=0$ and \eqref{1.2} holds.

{\rm (iii$^\prime$)} $\CAt=A_0$ if and only if $\cB_\tau >0$.
\end{theorem}
We prove also the following theorem which characterizes extensions $\wt A\in\Se$ with self-adjoint compression $C (\wt A)$.
\begin{theorem}\label{th1.3}
Let the assumptions be the same as in Theorem \ref{th1.2}. Then:

{\rm (i)} $\CAt$ is self-adjoint if and only if  $\lim\limits_{y\to\infty} y \im (\tau(iy)h,h)<\infty$ for all $ h\in \ker \cB_\tau$.

{\rm (i)} $\CAt$ is self-adjoint and transversal with $A_0$ (that is, $\CAt\cap A_0=A$ and $\CAt\wh + A_0=A^*$) if and only if  $\lim\limits_{y\to\infty} y \im (\tau(iy)h,h)<\infty$ for all $ h\in \cH$. In this  case the corresponding parameter for $\CAt$ in the Krein formula is the operator $\cN_\tau'$ in $\cH$ given by $\cN_\tau'h=\lim\limits_{y\to\infty} \tau(iy)h,\;h\in\cH$. Moreover, in the sense of \eqref{1.5} $\CAt=A_{\cN_\tau}$ with $\cN_\tau=-\cN_\tau'$.
\end{theorem}
Finally, we extend statement (\romannumeral 4) of Theorem \ref{th1.1} to a certain class of finite-codimensional extensions $\wt A\in\Se$ of a symmetric linear relation $A$ with possibly infinite equal deficiency indices $n_\pm (A)$. In the case $n_+(A)=n_-(A)<\infty$ this result holds for any finite-codimensional extension $\wt A\in\Se$.

\section{Preliminaries}
\subsection{Notations}
The following notations will be used throughout the paper: $\gH$, $\cH$ denote separable  Hilbert spaces; $\B (\cH_1,\cH_2)$  is the set of all bounded linear operators defined on $\cH_1$ with values in  $\cH_2$;  $A\up \mathcal L$ is a restriction of the operator $A\in \B(\cH_1,\cH_2)$ onto the linear manifold $\mathcal L\subset\cH_1$; $P_\cL$ is the orthoprojection in $\gH$ onto the subspace $\cL\subset \gH$;  $\bC_+\,(\bC_-)$ is the open  upper (lower) half-plane  of the complex plane.

Recall that a linear manifold $T$ in the Hilbert space $\cH_0\oplus\cH_1$ ($\cH\oplus\cH$) is called a  linear relation from $\cH_0$ to $\cH_1$ (resp. in $\cH$). The set of all closed linear relations from $\cH_0$ to $\cH_1$ (in $\cH$) will be denoted by $\C (\cH_0,\cH_1)$ (resp. $\C(\cH)$). Clearly for each linear operator $T:\dom T\to\cH_1, \;\dom T\subset \cH_0,$ its ${\rm gr} T =\{\{f,Tf\}:f\in \dom T\} $  is a linear relation from $\cH_0$ to $\cH_1$. This fact enables one to consider an operator as a linear relation.

For a linear relation $T$ from $\cH_0$ to $\cH_1$  we denote by
\begin{gather*}
\dom T:=\{h_0\in\cH_0: \exists h_1\in\cH_1\;\; \{h_0,h_1\}\in T\}, \quad \ker T:=\{h_0\in\cH_0:  \{h_0,0\}\in T\} \\
\ran T:=\{h_1\in\cH_1: \exists h_0\in\cH_0\;\; \{h_0,h_1\}\in T\}, \quad \mul T:=\{h_1\in\cH_1: \{0,h_1\}\in T\}
\end{gather*}
the domain, kernel, range and multivalued part  of
$T$ respectively; moreover, we let
\begin{gather*}
 \wh\mul T=\{0\}\oplus\mul T=\{\{0,h_1\}: h_1\in\mul T\}.
\end{gather*}
Denote also by $T^{-1}$ and $T^*$ the inverse and adjoint linear relations of $T$ respectively.

For an operator $T=T^*\in \B(\gH)$ we write $T\geq 0$ if $(Tf,f)\geq 0,\; f\in \gH,$ and $T>0 $ if $T-\a I\geq 0$ with some $\a>0$.

For linear relations $T_1$ and $T_2$ in $\cH$ we let $T_1\wh + T_2=\{\wh f+\wh g:\wh f\in T_1,\, \wh g\in T_2\}$.
\subsection{Nevanlinna functions}
Recall that  a holomorphic operator function $M:\CR\to \B(\cH)$ is called a Nevanlinna function if $\im\l\cd\im M (\l)\geq 0$ and $M^*(\l)=M (\ov\l), \; \l\in\CR$. The class of all  Nevanlinna $\B(\cH)$-valued functions will be denoted by $R[\cH]$.
\begin{definition}\label{def2.1}
The operator-function $M\in R[\cH]$ is referred to the class:

{\rm (\romannumeral 1)} $R_c[\cH]$, if $\ran \im M(\l)$ is closed for all $\l\in\CR$;

{\rm (\romannumeral 2)} $R_u[\cH]$, if $(\im M(\l))^{-1}\in\B (\cH)$ for all $\l\in\CR$ or equivalently if $\im\l\cd\im M (\l) > 0, \; \l\in\CR$.
\end{definition}
Clearly $R_u[\cH]\subset R_c[\cH]\subset R[\cH]$; moreover in the case $\dim \cH<\infty $ one has  $R_c [\cH]=R [\cH]$.

The following proposition is well known (see e.g. \cite{Mal92}).
\begin{proposition}\label{pr2.1.1}
If $M\in R[\cH]$, then the equality
\begin{gather}\label{2.3}
\cB_M=s\text{-}\lim_{y\to\infty}\tfrac 1 {iy}M(iy)
\end{gather}
defines the operator $\cB_M\in \B(\cH)$ such that $\cB_M\geq 0$. Moreover, the equality
\begin{gather}\label{2.3.1}
\dom \cN_M =\{h\in\cH: \lim_{y\to\infty}  y\im (M(iy)h,h)< \infty\}
\end{gather}
defines the (not necessarily closed) linear manifold $\dom\cN_M \subset \cH$ and for each $h\in \dom \cN_M $ there exists the limit
\begin{gather}\label{2.3.2}
\cN_M h:=\lim_{y\to\infty}  M(iy)h, \quad h\in \dom \cN_M.
\end{gather}
Hence the equalities \eqref{2.3.1} and \eqref{2.3.2} define the linear operator $\cN_M:\dom \cN_M\to \cH$.
\end{proposition}
\begin{proposition}\label{pr2.2}
Let $\tau\in R[\cH]$. Then the subspace $\cH'':=\ker\im \tau(\l) \subset \cH$ does not depend on $\l\in\CR$ and the block representation
\begin{gather}\label{2.4}
\tau(\l)=\begin{pmatrix} \tau_1(\l) & B_1 \cr B_1^* & B_2 \end{pmatrix}:\cH'\oplus\cH''\to \cH'\oplus\cH'', \quad \l\in\CR
\end{gather}
holds with $\cH'=\cH\ominus\cH'',\; B_2=B_2^* \in\B (\cH'')$, $B_1\in\B (\cH'',\cH')$ and the operator-function  $\tau_1(\cd)\in R[\cH']$ such that $\ker \im \tau_1(\l)=\{0\}, \; \l\in \CR$. Moreover, $\tau\in R_c[\cH]$ if and only if $\tau_1\in R_u[\cH']$.
\end{proposition}
\begin{proof}
As is known the Cayley transform $K(\l)=(\tau(\l)-i)(\tau(\l)+i)^{-1},\;\l\in\bC_+,$  is a holomorphic contractive operator-function on $\bC_+$ and for each $\l\in\bC_+$ the operator  $\tau(\l)$ admits the representation
\begin{gather}\label{2.5}
f = i(K(\l)-I)h, \quad \tau(\l)f= (K(\l)+I)h, \quad h\in\cH.
\end{gather}
It follows from \eqref{2.5} that
\begin{gather}\label{2.6}
\im (\tau(\l)f,f)=||h||^2 - ||K(\l)h||^2
\end{gather}
for all $\l\in\bC_+, \; h\in\cH$ and $ f=i(K(\l)-I)h$. According to \cite{NaFo} the subspace $\cH_0=\{h\in\cH: ||K(\l)h||=||h||\}$ and the operator $V=K(\l)\up \cH_0 (\in \B(\cH_0,\cH))$ do not depend on $\l\in\bC_+$. Moreover, by \eqref{2.5} and \eqref{2.6} $\ker \im \tau (\l)=(V-I)\cH_0$ and hence the subspace $\cH'':=\ker \im \tau (\l)$ does not depend on $\l\in\bC_+$.  Note  also that by \eqref{2.5} the operator $\tau(\l)\up \cH'' $ is defined by
\begin{gather*}
f = i(V-I)h_0, \quad \tau(\l)f= (V+I)h_0, \quad h_0\in\cH_0.
\end{gather*}
Therefore $B=\tau(\l)\up \cH'' (\in \B (\cH'',\cH))$ does not depend on $\l\in\bC_+$ and, consequently, $\tau(\l)$ admits the block representation
\begin{gather}\label{2.7}
\tau(\l)=\begin{pmatrix} \tau_1(\l) & B_1 \cr \tau_2(\l) & B_2 \end{pmatrix}:\cH'\oplus\cH''\to \cH'\oplus\cH'', \quad \l\in\bC_+
\end{gather}
This implies that
\begin{gather*}
\im \tau(\l)=\begin{pmatrix} \im\tau_1(\l) & \frac 1 {2i} (B_1-\tau_2^*(\l) \cr \frac 1 {2i} ( \tau_2(\l)-B_1^*) &  \im B_2 \end{pmatrix}:\cH'\oplus\cH''\to \cH'\oplus\cH'', \quad \l\in\bC_+
\end{gather*}
and hence $\im B_2 =0, \; \tau_1(\l)=B_1^*$. Therefore by \eqref{2.7} the equality \eqref{2.4} holds with $B_2=B_2^*$, which implies that
\begin{gather}\label{2.8}
\im \tau(\l)=\begin{pmatrix} \im\tau_1(\l) &0 \cr 0 & 0 \end{pmatrix}:\cH'\oplus\cH''\to \cH'\oplus\cH'', \quad \l\in\bC_+.
\end{gather}
Hence $\ker \im \tau_1(\l)=\{0\}$ and the required statements are proved for $\l\in\bC_+$. The same statements for $\l\in\bC_-$ are implied by $\tau(\l)=\tau^*(\ov\l), \; \l\in\bC_-$.
\end{proof}
As is known (see e.g. \cite{DM09}) a function $\tau:\CR\to \C (\cH)$ is referred to the class $\RH$ of Nevanlinna relation valued functions if:

(\romannumeral 1) $\im (f',f)\geq 0, \quad \{f,f'\}\in\tau (\l),\quad \l\in\bC_+$;

(\romannumeral 2)  $(\tau(\l)+i)^{-1}\in \B (\cH), \; \l\in\bC_+,$ and $(\tau (\l)+i)^{-1}$ is a holomorphic operator-function in $\bC_+$;

(\romannumeral 3) $\tau^*(\l)=\tau (\ov\l),\;\l\in\CR$.

A function $\tau\in \wt R(\cH)$ is referred to the class $R(\cH)$ if its values are operators, i.e., if $\mul \tau(\l)=0, \; \l\in\CR$.
If $\tau(\cd)\in R(\cH)$, then $\ov {\dom \tau(\l)}=\cH,\; \l\in\CR$.

 It is clear that $R [\cH]\subset R(\cH)\subset \wt R(\cH)$.

According to \cite{KreLan71} for each function $\tau\in \RH$ the multivalued part $\cK:=\mul \tau(\l)$ of $\tau (\l)$ does not depend on $\l\in\CR$ and the decompositions
\begin{gather}\label{2.9}
\cH=\cH_0 \oplus\cK, \qquad \tau (\l)=\tau_0(\l)\oplus \wh \cK, \quad \l\in\CR
\end{gather}
hold with $\tau_0\in R [\cH_0]$ and $\wh\cK=\wh \mul \tau(\l)=\{0\}\oplus \cK$. The operator function $\tau_0$ is called the operator part of $\tau$.
\begin{definition}\label{def2.3}
A relation valued function $\tau\in \RH$ is referred to the class $\RCH$  ($\wt R_u(\cH)$) if  decompositions  \eqref{2.9} hold with $\tau_0\in R_c [\cH_0]$ (resp. with $\tau_0\in R_u [\cH_0]$).
\end{definition}
Clearly $\wt R_u(\cH)\subset \wt R_c(\cH)\subset \wt R(\cH)$ and  in the case $\dim \cH<\infty$ one has $\RCH=\RH$.

The following corollary is immediate from Proposition \ref{pr2.2}.
\begin{corollary}\label{cor2.4}
The relation valued function $\tau:\CR\to \C (\cH)$ belongs to the class $\RCH$ if and only if there exist a decomposition
\begin{gather}\label{2.9.1}
\cH=\cH_0\oplus\cK=\underbrace{\cH'\oplus\cH''}_{\cH_0}\oplus\cK,
\end{gather}
operators $B_1\in\B(\cH'',\cH')$,  $B_2=B_2^*\in\B (\cH'')$ and an operator-function $\tau_1\in R_u[\cH']$ such that $\tau(\l)$ admits the representation \eqref{2.9} with the operator function
\begin{gather}\label{2.9.1.1}
\tau_0(\l)=\begin{pmatrix} \tau_1(\l) & -B_1\cr -B_1^* & -B_2\end{pmatrix}: \underbrace{\cH'\oplus\cH''}_{\cH_0}\to \underbrace{\cH'\oplus\cH''}_{\cH_0}, \quad \l\in\CR.
\end{gather}
 This means that
\begin{gather}\label{2.9.2}
\tau(\l)=\{\{h\oplus\f,(\tau_1(\l)h- B_1 \f)\oplus (-B_1^*h - B_2\f)\oplus k\}:\, h\oplus\f\in \cH'\oplus\cH'', \, k\in\cK\}.
\end{gather}
\end{corollary}
\begin{remark}\label{rem2.4a}
In the following for the relation valued function $\tau\in\RCH$ represented in accordance with Corollary \ref{cor2.4} by \eqref{2.9.1} and \eqref{2.9.2} we will use the notation $\tau =\{\cH'\oplus\cH''\oplus\cK, B_1,B_2, \tau_1\}$.
\end{remark}

\subsection{Boundary triplets and Weyl function}
Recall that a linear relation $T$ in $\cH$ is called: (i)~symmetric
if $T\subset T^*$ or, equivalently, if
\begin{gather}\label{2.9.2.1}
(f',g)-(f,g')=0, \quad \{f,f'\}, \{g,g'\}\in T;
\end{gather}

(ii)~self-adjoint if $T=T^*$ (and hence $T\in \C (\cH)$).

Assume that $\t$ is a symmetric relation in $\cH$ with a closed multivalued part  $\mul\t$. Then
\begin{gather}\label{2.9.3}
\cH=\cH_0\oplus\cK, \qquad \t=B\oplus \wh\mul \t=\{\{h, Bh+k\}:h\in \dom\t, \, k\in\mul\t\},
\end{gather}
where $\cK=\mul\t,\; \cH_0=\cH\ominus \mul\t$ and $B$ is a  symmetric operator in $\cH_0$ (the operator part of $\t$) with  $\dom B=\dom\t\subset \cH_0$. Clearly, $\t$ is closed (self-adjoint) if and only if $B$ is closed (self-adjoint).

The following lemma will be useful in the sequel.
\begin{lemma}\label{lem2.4b}
Let $\t$ be a symmetric linear relation in $\cH$. Then:

{\rm (i)} $\; \dom\t\subset \cH\ominus\mul\t$.

{\rm (ii)} If $\ov{\mul\t}=\mul\t$ and $\cH\ominus \mul\t\subset \dom\t$, then $\t^*=\t$.

{\rm (iii)} If $\dim\cH<\infty$, then the condition $\cH\ominus \mul\t\subset \dom \t$ is equivalent to  $ \t^*=\t$.
\end{lemma}
\begin{proof}
(i) Statement (i) directly follows from \eqref{2.9.2.1}

(ii) Assume that $\ov{\mul\t}=\mul\t$ and $\cH\ominus \mul\t\subset \dom\t$. Then in the representation \eqref{2.9.3} of $\t$ one has $\dom B=\dom\t=\cH_0$. Therefore  $B=B^*\in\B(\cH_0)$ and, consequently, $\t^*=\t$.

(iii) Assume that  $\dim \cH <\infty$. Then $\mul\t$ is closed and hence decompositions \eqref{2.9.3} hold. Moreover, $\dim\cH_0<\infty $ and therefore
$\t^*=\t$ if and only if $\dom B=\cH_0$ or, equivalently, $\dom \t=\cH\ominus \mul\t$. This yields statement (ii).
\end{proof}
In the following we denote by $A$ a closed symmetric linear relation (in particular closed not necessarily densely defined symmetric  operator) in a Hilbert space $\gH$. Let  $\gN_\l(A)=\ker (A^*-\l)\; (\l\in\CR)$ be a defect subspace of $A$, let $\wh\gN_\l(A)=\{\{f,\l f\}:\, f\in
\gN_\l(A)\}$, let $\gN(A)=\gH\ominus\dom A (= \mul A^*)$ and let $n_\pm (A):=\dim \gN_\l(A)\leq\infty, \;
\l\in\bC_\pm,$ be deficiency indices of $A$. Denote by $\ex $ the
set of all proper extensions of $A$ (i.e., the set of all
relations $\wt A$ in $\gH$ such that $A\subset\wt A\subset
A^*$) and  by $\cex$ the set of closed  extensions  $\wt A\in\ex $. Recall that two extensions $\wt A_1,\wt A_2 \in\cex$ are called transversal if $\wt A_1\cap \wt A_2=A$ and $\wt A_1\wh + \wt A_2=A^*$.
\begin{definition}\label{def2.5}$ \,$\cite{GorGor,Bru77}
A collection $\Pi=\bt$ consisting of a Hilbert space $\cH$ and linear mappings   $\G_j:A^*\to \cH, \; j\in\{0,1\},$  is called a
boundary triplet for $A^*$, if the mapping $\G=(\G_0,\G_1)^\top $ from $A^*$ into
$\cH\oplus\cH$ is surjective and the following  Green's
identity  holds:
\begin {equation}\label{2.10}
(f',g)-(f,g')=(\G_1  \wh f,\G_0 \wh g)- (\G_0 \wh f,\G_1
\wh g), \quad \wh f=\{f,f'\}, \; \wh g=\{g,g'\}\in A^*.
\end{equation}
\end{definition}
\begin{proposition}\label{pr2.5} $\,$ \cite{GorGor,Mal92} If $\Pi=\bt$ is a boundary triplet for $A^*$, then $n_+(A)=n_-(A)=\dim\cH$. Conversely, for each symmetric relation $A$ in $\gH$ with equal deficiency indices  $n_+(A)=n_-(A)$ there exists a boundary triplet for $A^*$.
\end{proposition}
\begin{proposition}\label{pr2.6}$\,$ \cite{GorGor,Mal92}
Let $\Pi=\bt$ be a boundary triplet for  $A^*$. Then:

{\rm (\romannumeral 1)}   $\ker \G=A$ and $\G$ is a bounded
operator from $A^*$ onto $\cH\oplus\cH$.

{\rm (\romannumeral 2)}   The
mapping
\begin {equation}\label{2.10.1}
\t\to  A_\t :=\{ \hat f\in A^*:\{\G_0  \hat f,\G_1 \hat f \}\in
\t\}
\end{equation}
establishes a bijective correspondence between all linear
relations  $\t$ in $\cH$ and all extensions $ \wt A= A_\t\in\ex$.
In the case $\t\in \B(\cH)$ one has $ A_\t=\{ \hat f\in A^*:\G_1 \hat f = \t  \G_0  \hat f\}$

{\rm (\romannumeral 3)} $(A_\t)^*=A_{\t^*}$

{\rm (\romannumeral 4)} $\ov A_\t=A_{\ov\t}$ and hence $A_\t\in\cex$ if and only if $\t \in\C (\cH)$

{\rm (\romannumeral 5)} $ A_\t$ is  symmetric (self-adjoint) if  and only if $\t$ is symmetric (resp. self-adjoint)

{\rm (\romannumeral 6)} The equality
\begin {equation}\label{2.11}
A_0:=\ker\G_0=\{\hat f\in A^* :\G_0\hat f=0\}
\end{equation}
defines a self-adjoint extension $A_0$  of $A$.

{\rm (\romannumeral 7)} The extensions $A_0$ and $A_\t\in \cex$ are transversal if and only if $\t\in\B (\cH)$.
\end{proposition}
The following lemma will be useful in the sequel.
\begin{lemma}\label{lem2.7}
Assume that $\Pi=\bt$ is a boundary triplet for $A^*$,
\begin{gather}\label{2.19}
\cH=\cH'\oplus\cH''\oplus\cK
\end{gather}
and the block representations  of $\G_0$ and $\G_1$ are
\begin{gather}\label{2.20}
\G_0=\begin{pmatrix}\G_{01} \cr  \G_{02}\cr \G_{03}\end{pmatrix} :A^*\to \cH'\oplus\cH''\oplus\cK, \qquad \G_1=\begin{pmatrix}\G_{11} \cr  \G_{12}\cr \G_{13}\end{pmatrix}:A^*\to \cH'\oplus\cH''\oplus\cK.
\end{gather}
Moreover, let $B\in\cB(\cH'',\cH'\oplus\cH'')$ be the operator with the block representation
\begin{gather}\label{2.20.0}
B=\begin{pmatrix} B_1 \cr B_2 \end{pmatrix}:\cH''\to \cH'\oplus\cH''
\end{gather}
such that $B_2=B_2^*$ (this means that $B$ is a bounded symmetric operator in $\cH'\oplus\cH''$ with the domain $\dom B=\cH''$ )and let $\t_0\in \C (\cH)$ be given by
\begin{gather}\label{2.20.1}
\t_0=B\oplus\wh \cK=\{\{h,B_1h\oplus B_2h\oplus k \}:h\in\cH'',\, k\in\cK\}
\end{gather}
with $\wh \cK=\wh\mul \t_0=\{0\}\oplus\cK$.
Then:

{\rm (\romannumeral 1)}The equalities
\begin{gather}
S:=A_{\t_0}=\{\wh f \in A^*:\G_{01}\wh f=0, \, \G_{03}\wh f=0, \,\G_{11}\wh f=B_1 \G_{02}\wh f, \,\G_{12}\wh f=B_2 \G_{02}\wh f \}\label{2.21}\\
S^*=\{\wh f \in A^*: \G_{03}\wh f=0, \,\G_{12}\wh f=B_1^* \G_{01}\wh f + B_2 \G_{02}\wh f \}\label{2.22}
\end{gather}
define a symmetric extension  $S\in\cex$ and its adjoint $S^*$.

{\rm (\romannumeral 2)} The collection  $\Pi'=\{\cH',\G_0',\G_1'\}$ with operators $\G_j':S^*\to \cH'$ defined by
\begin{gather}\label{2.23}
\G_0' \wh f=\G_{01}\wh f, \qquad \G_1' \wh f=\G_{11}\wh f -B_1 \G_{02}\wh f, \quad \wh f\in S^*
\end{gather}
is a boundary triplet for $S^*$.

{\rm (\romannumeral 3)} Let $\t$ be a linear relation in $\cH'$ defined by
\begin{gather*}
\t=\t'\dotplus \wh\mul \t= \{\{h,\t' h+ h'\}: h\in\dom \t',\, h'\in\mul \t\},
\end{gather*}
with a linear operator $\t':\dom \t' \to \cH'\; (\dom\t' \subset \cH')$ and let the extension  $\wt A\in {\rm ext}(S)$ be given by  $\wt A=S_\t $ (in the triplet $\Pi'$). Moreover, let $\wt\t':\dom \t' \oplus\cH''\to \cH'\oplus\cH''$ be a linear operator, defined by
\begin{gather*}
\wt \t'= \begin{pmatrix} \t' & B_1 \cr B_1^*\up \dom \t' & B_2 \end{pmatrix}:\dom \t'\oplus\cH''\to \cH'\oplus\cH''
\end{gather*}
and let $\wt \t$ be a linear relation in $\cH$ given by
\begin{gather}
\mul \wt \t=\mul \t \oplus\cK\label{2.23a}\\
\wt\t:=\wt \t' \dotplus \wh\mul \wt\t=\{\{h\oplus\f, (\t'h +B_1\f+h')\oplus (B_1^*h +B_2\f)\oplus k\}:\qquad\qquad \label{2.23.0}\\
\qquad\qquad\qquad\qquad\qquad\qquad h\oplus\f\in\dom \t'\oplus\cH'',h'\in \mul\t , k\in\cK\}.\nonumber
\end{gather}
Then $\wt A \in \ex$ and  $\wt A =A_{\wt\t} $ (in the triplet $\Pi$).

If in addition $\t\in\B (\cH')$, then
\begin{gather*}
\wt \t'= \begin{pmatrix} \t & B_1 \cr B_1^* & B_2 \end{pmatrix}:\cH'\oplus\cH''\to \cH'\oplus\cH''.
\end{gather*}
and \eqref{2.23.0} takes the form
\begin{gather}\label{2.23.0.1}
\wt\t=\wt \t'\oplus\wh\cK=\{h\oplus\f, (\t +B_1\f)\oplus (B_1^*h +B_2\f)\oplus k\}:h\oplus\f\in\cH'\oplus\cH'', k\in\cK\}
\end{gather}
\end{lemma}
\begin{proof}
(i) Clearly  the relation $\t_0$ is symmetric and by Proposition \ref{pr2.6},  (\romannumeral 5) $S:=A_{\t_0}\in \cex$ is symmetric. Moreover, by \eqref{2.20} the second equality in \eqref{2.21} holds. Since obviously
\begin{gather*}
\t_0^*=\{\{h_1\oplus h_2, h'\oplus (B_1^* h_1+B_2 h_2) \oplus  k\}:h_1\oplus h_2\in \cH'\oplus \cH'', h'\in\cH',k\in\cK \}
\end{gather*}
and by Proposition \ref{pr2.6},  (\romannumeral 3) $S^*=A_{\t_0^*}$, the equality \eqref{2.22} is valid.

(ii) Let $\wh f=\{f,f'\}, \wh g=\{g,g'\}\in S^*$. By using Green identity \eqref{2.10} (for the triplet $\Pi$) and \eqref{2.22} one obtains
\begin{gather*}
(f',g)-(f,g')=(\G_{11}\wh f, \G_{01}\wh g)-(\G_{01}\wh f, \G_{11}\wh g)+ (\G_{12}\wh f, \G_{02}\wh g)-(\G_{02}\wh f, \G_{12}\wh g)+\\
(\G_{13}\wh f, \G_{03}\wh g)-(\G_{03}\wh f, \G_{13}\wh g)=(\G_{11}\wh f, \G_{01}\wh g)-(\G_{01}\wh f, \G_{11}\wh g)+(\G_{01}\wh f,B_1\G_{02}\wh g)+\\
(B_2 \G_{02}\wh f,\G_{02}\wh g)- (B_1\G_{02}\wh f, \G_{01}\wh g)-(\G_{02}\wh f,B_2 \G_{02}\wh g)=((\G_{11} -B_1 \G_{02})\wh f,\G_{01}\wh g )-\\
(\G_{01}\wh f,(\G_{11} -B_1 \G_{02})\wh g)=(\G_1'\wh f, \G_0'\wh g)-(\G_0'\wh f, \G_1'\wh g).
\end{gather*}
This proves  Green's identity \eqref{2.10} for operators $\G_0'$ and $\G_1'$.

Next assume that $h_1', h_2'\in \cH'$. Since the mapping $\G=(\G_0,\G_1)^\top $ is surjective, there is $\wh f \in A^*$ such that
\begin{gather*}
\G_{01}\wh f=h_1', \;\;\;\G_{02}\wh f= 0,\;\;\; \G_{03}\wh f= 0,\;\;\; \G_{11}\wh f =h_2', \;\;\; \G_{12}\wh f= B_1^* h_1', \;\;\;\G_{13}\wh f= 0.
\end{gather*}
Hence $\G_{12}\wh f=B_1^* \G_{01}\wh f + B_2 \G_{02}\wh f $ and by \eqref{2.22} $\wh f\in S^*$. Moreover, in view of \eqref{2.23} one has $\G_0'\wh f=h_1', \;  \G_1'\wh f=h_2'$, which proves surjectivity  of the mapping $\G'=(\G_0',\G_1')^\top$.

(iii) It follows from \eqref{2.22} and \eqref{2.23} that $\wh f\in \wt A$ if and only if $\wh f\in A^*$ and
\begin{gather*}
\G_{01}\wh f \in \dom\t,\quad \G_{11}\wh f - B_1 \G_{02}\wh f =\t' \G_{01}\wh f +h', \quad \G_{12}\wh f=B_1^* \G_{01}\wh f + B_2 \G_{02}\wh f,\quad \G_{03}\wh f =0
\end{gather*}
with some $h'\in\mul\t$. In turn, these conditions are equivalent to inclusions $\wh f \in A^*$ and $\{\G_0\wh f, \G_1 \wh f\}\in \wt \t$. Hence $\wt A=A_{\wt \t}$ (in the triplet $\Pi$).
\end{proof}
\begin{lemma}\label{lem2.7.1}
Let $L_1$ and $L_2$ be closed subspaces in a Hilbert space $\cH$ such that $L_1\cap L_2=\{0\}$. Then $\ov{L_1\dotplus L_2}=L_1\dotplus L_2$ if and only if there exists a closed subspace $L_1'\supset L_1$ in  $\cH$ such that $\cH= L_1'\dotplus L_2$.
\begin{proof}
Let $\ov{L_1\dotplus L_2}=L_1\dotplus L_2$. Then $\cH=(L_1\dotplus L_2)\oplus \cH'$ with $\cH'=\cH\ominus (L_1\dotplus L_2)$ and hence $\cH=L_1'\oplus L_2$ with a closed subspace $L_1'=L_1\oplus\cH'$. Conversely, let $L_1'\supset L_1$  be a closed subspace in  $\cH$ such that $\cH= L_1'\dotplus L_2$ and  let $\pi_1\in \B (\cH,L_1')$ and $\pi_2\in \B (\cH,L_2)$ be the corresponding skew projections onto $L_1'$ and $L_2$ respectively. Assume that $f_n\in L_1\dotplus L_2$ and $f_n\to f$. Then $\pi_1 f_n\in L_1, \; \pi_1 f_n \to \pi_1 f$ and hence $\pi_1 f\in L_1$. Since $\pi_2 f\in L_2$, it follows that $f(=\pi_1 f+\pi_2 f)\in L_1\dotplus L_2$. Therefore $L_1\dotplus L_2$ is closed.
\end{proof}
\end{lemma}
\begin{proposition}\label{pr2.7.2}
Let $\Pi=\bt$ be a boundary triplet for $A^*$, let $A_0=\ker \G_0$  and let $\t\in \C (\cH)$ be a linear relation such that decompositions \eqref{2.9.3} hold with a (closed) linear operator $B:\dom\t\to \cH_0, \; \dom \t\subset \cH_0$ (in particular, this assumption is satisfied for symmetric $\t$). Then $\ov{A_\t\wh + A_0}=A_\t\wh + A_0$ if and only if $B\in\B (\dom \t,\cH_0)$ (that is, if and only if $\ov{\dom\t}=\dom\t$).
\end{proposition}
\begin{proof}
Let $\wh\cH=\{0\}\oplus \cH$ and $\wh\cH_0=\{0\}\oplus \cH_0$. Since $A_0=A_{\wh\cH}$, it follows from Proposition \ref{pr2.6}, (i) that
\begin{gather}\label{2.24}
\ov{A_\t\wh + A_0}=A_\t\wh + A_0 \iff \ov{\t\wh + \wh\cH}=\t\wh + \wh\cH.
\end{gather}
Moreover, $B\cap \wh\cH_0=\{0\}$ and $\wh\cH=\wh\cH_0 \oplus \wh\cK$, which in view of \eqref{2.9.3} implies that $\t\wh + \wh\cH=(B\dotplus \wh\cH_0)\oplus \wh\cK$. Therefore
\begin{gather}\label{2.24.1}
\ov{\t\wh + \wh\cH}=\t\wh + \wh\cH\iff \ov{B\dotplus \wh\cH_0}=B\dotplus \wh\cH_0.
\end{gather}
Next, by Lemma \ref{lem2.7.1} $\ov{B\dotplus \wh\cH_0}=B\dotplus \wh\cH_0$  if and only if there exists $B'\in\C (\cH_0)$ such that $B\subset B'$ and $B'\oplus \wh\cH_0=\cH_0^2$. Since the last equality is equivalent to the inclusion $B'\in\B(\cH_0)$, the following equivalence holds:
\begin{gather}\label{2.24.2}
\ov{B\dotplus \wh\cH_0}=B\dotplus \wh\cH_0\iff B\in \B(\dom \t,\cH_0).
\end{gather}
Now combining \eqref{2.24}, \eqref{2.24.1} and \eqref{2.24.2}, one obtains the required statement.
\end{proof}
\begin{corollary}\label{cor2.7.3}
Let $\Pi=\bt$ be a boundary triplet for $A^*$, let $A_0=\ker \G_0$ and let $\t\in \C (\cH)$. Then $A_\t \cap A_0=A$ and  $\ov{A_\t\wh + A_0}=A_\t\wh + A_0$ if and only if $\t\in\B (\dom \t,\cH)$.
\end{corollary}
\begin{proof}
Since $A_0=A_{\{0\}\oplus\cH}$, it follows from Proposition \ref{pr2.6}, (i) that $A_\t \cap A_0=A$ if and only if $\t\cap (\{0\}\oplus\cH)=\{0\}$, that is if and only if $\t$ is an operator. Thus $A_\t \cap A_0=A$ if and only if decompositions \eqref{2.9} hold with $\cK=\{0\}$. This and Proposition \ref{pr2.7.2} yield the result.
\end{proof}
In the rest of this subsection we recall some definitions and results from \cite{DM91,Mal92,DM95}.
\begin{proposition}\label{pr2.8}
Let $\Pi=\bt$ be a boundary triplet for $A^*$ and let $A_0=A_0^*\in \cex$ be given by \eqref{2.11}. Moreover, let $\pi_1$ be the
orthoprojection  in $\gH\oplus\gH$ onto $\gH\oplus \{0\}$. Then the operator $\G_0\up \wh \gN_\l (A), \;\l\in\CR,$  isomorphically maps $\wh\gN_\l (A)$ onto $\cH$ and hence  the equalities
\begin{gather*}
\g (\l)=\pi_1(\G_0\up\wh \gN_\l (A))^{-1}, \qquad \G_1\up\wh \gN_\l(A)=M(\l)\G_0\up\wh \gN_\l(A), \quad \l\in\CR
\end{gather*}
correctly define the operator functions $\g(\cdot):\CR\to \B(\cH,\gH)$ and $M(\cdot):\CR\to \B(\cH)$. Moreover, $\g(\cd)$ and $M(\cd)$ satisfy the identities
\begin{gather}
\g(\l)=\g(z)+(\l -z)(A_0-\l)^{-1}\g(z)\label{2.25}\\
 M(z)-M^*(\l)=(z-\ov\l)\g^*(\l)\g(z),  \quad z,\l \in\CR \label{2.25.1}
\end{gather}
which imply that $\g(\cd)$ and $M(\cd)$  are holomorphic in $\CR$ and $M(\cd)\in R[\cH]$.
\end{proposition}
\begin{definition}\label{def2.9}
The operator-functions $\g(\cd)$ and $M(\cd)$ defined in
Proposition \ref{pr2.8}  are called the
$\g$-field and the Weyl function of the triplet $\Pi$ respectively.
\end{definition}
\begin{remark}\label{rem2.10}
It follows from \eqref{2.25} and \eqref{2.25.1} that $\g(\cd)$ and $M(\cd)$ are the $\g$-field  and the $Q$-function of the pair $(A,A_0)$ respectively in the sense of  \cite{KreLan73,LanTex77}.
\end{remark}
\begin{definition}\label{def2.11}
The (not necessarily closed) linear relation $\cF$ in $\cH$ defined by
\begin{gather*}
\cF:=\G \wh\mul A^*= \{\{\G_0\{0,n\},\G_1\{0,n\}\}:n\in \gN (A)\}
\end{gather*}
is called a forbidden relation of the boundary triplet $\Pi=\bt$ for $A^*$.
\end{definition}
In the following theorem the forbidden relation $\cF$ is characterized in terms of the asymptotic behavior of the Weyl function.
\begin{theorem}\label{th2.12}
Let $\Pi=\bt$ be a boundary triplet for $A^*$, let $M(\cd)$ be the Weyl function of $\Pi$ and let $B_M$ and $\cN_M$ be the operators defined by  \eqref{2.3} and \eqref{2.3.1}, \eqref{2.3.2} respectively. Moreover, let  $\cF$ be the forbidden relation of $\Pi$. Then
\begin{gather}\label{2.27}
\ran B_M\subset \mul \cF, \quad \ov{\ran B_M} = \ov{\mul \cF},
\end{gather}
and $\cF$ admits the representation
\begin{gather}\label{2.28}
\cF=\cN_M \dotplus \wh\mul \cF= \{\{h,\cN_M h+h'\}: h\in\dom \cN_M,\, h'\in \mul \cF\}.
\end{gather}
\end{theorem}
\subsection{Exit space extensions and formula for generalized resolvents}
As is known a linear relation $\wt A=\wt A^*$ in a Hilbert space $\wt\gH\supset \gH$ is called an exit space extension of $A$ if $A\subset \wt A$ and the  minimality condition $\ov{{\rm span}} \{\gH,(\wt A-\l)^{-1}\gH: \l\in\CR\}=\wt\gH$ is satisfied. For an exit space extension $\wt A\in\C (\wt\gH)$ of $A$ the compressed resolvent
\begin{gather}\label{2.30}
R(\l)=P_\gH (\wt A -\l)^{-1}\up\gH,\quad\l\in\CR
\end{gather}
is called a generalized resolvent of $A$ (here $P_\gH$ is the orthoprojection in $\wt\gH$ onto $\gH$). If two exit space extensions $\wt A_1\in\C (\wt\gH_1)$ and $\wt A_2\in\C (\wt\gH_2)$ of $A$ generates the same generalized resolvent $R(\l)$, then $\wt A_1$ and $\wt A_2$ are equivalent. The latter means that there exists a unitary operator $V\in\B (\wt\gH_1\ominus\gH,\wt\gH_2\ominus\gH  )$ such that $\wt A_2=\wt U \wt A_1$ with the unitary operator $\wt U=(I_\gH\oplus V)\oplus (I_\gH\oplus V)\in B(\wt\gH_1^2,\wt \gH_2^2)$. Hence each exit space extension $\wt A$ of $A$ is defined by the generalized resolvent \eqref{2.30} uniquely up to the equivalence.
\begin{theorem}\label{th2.14}$\,$\cite{DM91,Mal92}
Assume that $\Pi=\bt$ is a boundary triplet for  $A^*$, $A_0=\ker \G_0$ and $\g(\cd)$ and $M(\cd)$ are the $\g$-field and the Weyl function of $\Pi$ respectively. Then: {\rm (i)} the equality (Krein formula for generalized resolvents)
\begin{gather}\label{2.31}
P_\gH (\wt A_\tau -\l)^{-1}\up\gH = (A_0-\l)^{-1}- \g(\l)(\tau (\l)+M(\l))^{-1} \g^*(\ov\l), \quad \l\in\CR
\end{gather}
establishes a bijective correspondence $\wt A=\wt A_\tau$ between all relation valued functions $\tau=\tau(\l)\in \RH$ and all exit space self-adjoint extensions $\wt A$ of $A$. Moreover, for each $\tau\in\RH$ the following equality holds:
\begin{gather}\label{2.32}
P_\gH (\wt A_\tau -\l)^{-1}\up\gH =(A_{-\tau(\l)}-\l)^{-1}, \quad \l\in\CR;
\end{gather}
{\rm (ii)} an extension $\wt A_\tau$ is canonical (that is, $\wt A_\tau\in \C (\gH)$) if and only if $\tau(\l)\equiv \t(=\t^*), \; \l\in\CR$. In this case $\wt A_\tau=A_{-\t}$ (in the sense of Proposition \ref{pr2.6}, {\rm (ii)}).
\end{theorem}
It follows from Theorem \ref{th2.14} that Krein formula \eqref{2.31} gives a parametrization $\wt A=\wt A_\tau$  of all exit space extensions $\wt A=\wt A^*$ of $A$ in terms of functions $\tau(\cd)\in\RH$. Without connection with boundary triplets formula \eqref{2.31} was originally proved in \cite{KreLan71,LanTex77}.
\section{Compressions of exit  space self-adjoint extensions}
\subsection{Parametrization of compressions}
Assume that $\wt \gH\supset\gH$ is a Hilbert space, $\gH_r:=\wt\gH\ominus\gH $,  $P_\gH$ is the orthoprojection in $\wt\gH$ onto $\gH$ and $\wt A=\wt A^*\in\C (\wt \gH)$ is an exit space extension of $A$. In the following we let
\begin{gather}
S(\wt A):=\wt A\cap \gH^2=\{\wh f\in\gH^2:\wh f \in \wt A\}\label{3.3}\\
T(\wt A):=\{\{P_\gH f, P_\gH f'\}:\{f,f'\}\in \wt A \}\label{3.3.1}\\
\CA:=P_{\gH}\wt A\up \gH=\{{\{f,f'\}\in \gH^2: \{f,f'\oplus f_r'\}\in \wt A}\;\;\text{\rm with some}\;\; f_r'\in\gH_r\}\label{3.4}
\end{gather}
Clearly, $S(\wt A)$ is a closed symmetric relation in $\gH$, $C(\wt A)$ is a symmetric relation in $\gH$, $T(\wt A)$ is a linear relation in $\gH$ and
\begin{gather}\label{3.4.1}
A\subset S(\wt A)\subset C(\wt A) \subset T(\wt A)\subset A^*.
\end{gather}
Moreover, $S(\wt A)\subset\wt A$ and according to \cite{Sht62,DM06} $(S(\wt A))^*=\ov {T(\wt A)}$.
\begin{definition}\label{def3.0}
The linear relation $\CA$ is called the compression of $\wt A$.
\end{definition}

The following theorem directly follows from the results of \cite{DM00,DM09}.
\begin{theorem}\label{th3.1}
Let $\Pi=\bt$ be a boundary triplet for  $A^*$, let $\tau\in \RH$ and let $\wt A_\tau=\wt A_\tau^*$ be the corresponding exit space extension of $A$. Then the equalities $S(\wt A)=A$ and $\ov{T(\wt A)}= T(\wt A)$ hold if and only if $\tau\in R_u[\cH]$.
\end{theorem}
\begin{lemma}\label{lem3.2}
Let the assumptions of Lemma \ref{lem2.7} be satisfied. Moreover, let $S\in\cex$ be symmetric extension \eqref{2.21} of $A$, let $\Pi'=\{\cH',\G_0',\G_1'\}$ be  boundary triplet \eqref{2.23} for $S^*$, let $\tau_1(\cd)\in R_u[\cH']$ and let $\wt A =\wt S_{\tau_1}$ be the corresponding exit space self-adjoint  extension of $S$. Assume also that $\tau=\{\cH'\oplus\cH''\oplus\cK,B_1,B_2,\tau_1\}\in \wt R_c (\cH)$ (see Remark \ref{rem2.4a}). Then $\wt A\supset A$ and $\wt A=\wt A_\tau$ (in the triplet $\Pi$).
\end{lemma}
\begin{proof}
According to \eqref{2.32}
\begin{gather}\label{3.5}
P_\gH(\wt A-\l)^{-1} \up\gH=(\wt A(\l)-\l)^{-1}, \quad \l\in\CR,
\end{gather}
where $\wt A(\l)=S_{-\tau_1(\l)}$ (in the triplet $\Pi'$). Moreover,
\begin{gather*}
-\tau(\l)=\{h\oplus\f, (-\tau_1(\l)h +B_1\f)\oplus (B_1^*h +B_2\f)\oplus k\}:h\oplus\f\in\cH'\oplus\cH'', k\in\cK\}.
\end{gather*}
Comparing this equality  with \eqref{2.23.0.1} we obtain from Lemma \ref{lem2.7}, (iii) that $\wt A(\l)\in\cex$ and $\wt A(\l)=A_{-\tau(\l)}$ (in the triplet $\Pi$). Therefore by \eqref{3.5} and \eqref{2.32} $\wt A=\wt A_\tau$ (in the triplet $\Pi$).
\end{proof}
\begin{proposition}\label{pr3.3}
Let $\Pi=\bt$ be a boundary triplet for $A^*$, let $\tau\in \RCH$ and let $\wt A_\tau=\wt A_\tau^*$ be the corresponding exit space extension of $A$. Assume also  that
\begin{gather}\label{3.6}
\tau=\{\cH'\oplus\cH''\oplus\cK,B_1,B_2,\tau_1\}
\end{gather}
(see Remark \ref{rem2.4a}), $\G_0$ and $\G_1$ have the block representations \eqref{2.20} and  $\t_0\in \C (\cH)$  is given by  \eqref{2.20.1}. Then the extension $S=S(\wt A_\tau)$ admits the representation \eqref{2.21} (in the triplet $\Pi$).
\end{proposition}
\begin{proof}
Clearly the assumptions of Lemma \ref{lem2.7} are satisfied with the subspaces $\cH',\cH'',\cK$ and the operator $B=(B_1,B_2)^\top$. Let $S\in \cex $ be  symmetric extension \eqref{2.21} and let $\Pi'=\{\cH',\G_0',\G_1'\}$ be boundary triplet  \eqref{2.23} for $S^*$. Assume that $\wt A: =\wt S_{\tau_1}$ (in the triplet $\Pi'$). Since $\tau_1\in R_u [\cH']$, it follows from Theorem \ref{th3.1} that $S(\wt A)=S$. Moreover, according to Lemma \ref{lem3.2} $A\subset \wt A$ and $\wt A=\wt A_\tau$ (in the triplet $\Pi$), which proves the required statement.
\end{proof}
\begin{corollary}\label{cor3.4}
Let $\Pi=\bt$ be a boundary triplet for $A^*$, let $A_0=\ker\G_0$ and let $\tau\in \RCH$.  Then $S(\wt A_\tau)\cap A_0=A$ if and only if $\tau\in R_c[\cH]$.
\end{corollary}
\begin{proof}
Assume that $\tau$ is represented as in \eqref{3.6}. Then according to Proposition \ref{pr3.3} $S(\wt A_\tau)=A_{\t_0}$ with $\t_0$ given by \eqref{2.20.1}. Since obviously $A_{\t_0}\cap A_0=A_{\wh\cK}$, the equivalences $S(\wt A_\tau)\cap A_0=A \Leftrightarrow \cK=\{0\}\Leftrightarrow \tau \in R_c[\cH] $ are valid.
\end{proof}
In the following theorem we give a geometric characterization of self-adjoint extensions $\wt A_\tau$ corresponding to $\tau\in \RCH$.
\begin{theorem}\label{th3.5}
Assume that $\Pi=\bt$ is a boundary triplet for $A^*$ and $A_0=\ker\G_0$. Let $\tau\in \wt R(\cH)$ and let $\wt A_\tau $ be the corresponding exit space self-adjoint extension of $A$. Then:

{\rm (i)} The linear relations $ S(\wt A_\tau)\wh + A_0$ and $T(\wt A_\tau)$  are closed if and only if $\tau\in\RCH$.

{\rm (ii)} $S(\wt A_\tau)\cap A_0=A$ and the linear relations $ S(\wt A_\tau)\wh + A_0$ and $T(\wt A_\tau)$  are closed if and only if $\tau\in R_c[\cH]$.
\end{theorem}
\begin{proof}
(i) We put $S=S(\wt A_\tau)$ and $T=T(\wt A_\tau)$.  Assume that  $\ov{S\wh + A_0}= S\wh + A_0$ and $\ov{T}= T$. Since $S\in \cex$ and $S\subset S^*$, it follows from Proposition \ref{pr2.6}, (ii)  and (\romannumeral 5) that $S= A_{\t_0}$ with some symmetric relation $\t_0 \in \C (\cH)$. Moreover, by Proposition \ref{pr2.7.2} the decompositions
\begin{gather}\label{3.7}
\cH=\cH_0\oplus \cK, \qquad \t_0=B\oplus \wh\cK
\end{gather}
holds with $\cK=\mul\t_0, \; \wh\cK=\wh \mul\t_0$ and $B\in \B (\dom\t_0, \cH_0)$, where $\dom\t_0\subset \cH_0$ is closed. Let  $\cH''=\dom\t_0 $   and $\cH'=\cH_0\ominus \cH''$, so that $\cH_0=\cH'\oplus\cH''$. Then decomposition \eqref{2.19} of $\cH$ holds, the operator $B$ admits the block representation \eqref{2.20.0} and $\t_0$ can be written in the form \eqref{2.20.1}. Moreover, since $B$ is symmetric, the equality $B_2=B_2^*$ holds.  Therefore by Lemma \ref{lem2.7}, (ii) the equalities \eqref{2.23} define the boundary triplet $\Pi'=\{\cH',\G_0',\G_1'\}$ for $S^*$. Since $\wt A=\wt A_\tau$ is an extension of $S$, it follows from Theorem \ref{th3.1} that $\wt A=\wt A_{\tau_1}$ (in the triplet $\Pi'$) with  some $\tau_1 \in R_u [\cH']$. Therefore by Lemma \ref{lem3.2} $\tau\in \wt R_c(\cH)$.

Conversely, let $\tau\in \RCH$. Then by Corollary \ref{cor2.4} there exist a decomposition \eqref{2.9.1} of $\cH$, an operator $B=(B_1,B_2)^\top \in \B (\cH'', \cH'\oplus\cH'')$ and a function $\tau_1\in R_u[\cH']$ such that $\tau=\{\cH'\oplus\cH''\oplus\cK, B_1,B_2, \tau_1\}$. It follows from Proposition \ref{pr3.3} that $S=A_{\t_0}$ with $\t_0$ given by \eqref{2.20.1}. Therefore by Proposition \ref{pr2.7.2} $\ov{S\wh + A_0}= S\wh + A_0$. Next assume that $\Pi'=\{\cH',\G_0', \G_1'\}$ is a boundary triplet \eqref{2.23} for $S^*$. Then by Lemma \ref{lem3.2} $\wt A_\tau=\wt S_{\tau_1}$ (in the triplet $\Pi'$) and in view of Theorem \ref{th3.1} one has $\ov T= T$.

(ii) Combining statement (i) with Corollary \ref{cor3.4} we arrive at statement (ii).
\end{proof}
\begin{proposition}\label{pr3.7}
Assume that $\Pi=\bt$ is a boundary triplet for $A^*$, $\tau\in R_u[\cH]$ and $\wt A_\tau =\wt A_\tau^*$ is the corresponding exit space extension of $A$. Let $\wt A_\tau\in \C (\wt \gH)$ with a Hilbert space $\wt \gH\supset \gH$ and let $\gH_r=\wt\gH \ominus \gH$. Then:

{\rm (i)} There exist a symmetric relation $A_r\in\C (\gH_r)$ and a boundary triplet $\Pi_r=\{\cH,\G_0^r,\G_1^r\}$ for $A_r^*$ such that $\tau$ is the Weyl function of $\Pi_r$ and
\begin{gather}\label{3.10}
\wt A_\tau=\{\wh f\oplus\wh f_r\in A^*\oplus A_r^*: \G_0\wh f =\G_0^r \wh f_r, \,\G_1\wh f = -\G_1^r \wh f_r   \}.
\end{gather}

{\rm (ii)} $\CAt=A_{-\cF_r}$ (in the triplet $\Pi$), were   $\cF_r$ is the forbidden relation of $\Pi_r$.
\end{proposition}
\begin{proof}
Statement (i) directly follows from the results of \cite{DM00,DM09}.

(i) According to Proposition \ref{pr2.6}, (ii) statement (ii) is equivalent to the equality
\begin{gather}\label{3.11}
\CAt=\{\wh f\in A^*:  \{\G_0\wh f, -\G_1 \wh f\}\in \cF_r\}.
\end{gather}
Let $\wh f=\{f,f'\}\in\CAt$. Then $\wh f\in A^*$ and by \eqref{3.4} there exists $f_r'\in\gH_r$ such that $\wh g:=\{f,f'\oplus f_r'\} \in \wt A_\tau$. Letting $\wh f_r=\{0,f_r'\}\in\gH_r^2$ one gets $\wh g=\wh f\oplus \wh f_r$. Therefore $\wh f_r\in A_r^*$ and hence $f_r'\in\mul A_r^*$. Moreover, by \eqref{3.10} $\{\G_0\wh f, -\G_1\wh f\}=\{\G_0^r\wh f_r, \G_1^r\wh f_r\}\in \cF_r$.

Conversely, let $\wh f=\{f,f'\}\in A^*$ and $\{\G_0\wh f, -\G_1\wh f\}\in \cF_r$. Then there exists $f_r'\in \mul A_r^*$  such that $\G_0^r \{0,f_r'\}=\G_0\wh f$ and $\G_1^r \{0,f_r'\}=-\G_1\wh f$. Letting $\wh f_r=\{0,f_r'\}\in A_r^*$ we obtain from \eqref{3.10} that $\wh g:= \wh f\oplus \wh f_r=\{f,f'\oplus f_r'\} \in \wt A_\tau$ and therefore $\wh f=\{f,f'\}\in \CAt$. This proves \eqref{3.11}.
\end{proof}
In the following theorem we  parameterize in terms of $\tau$  the compressions $\CAt$ of exit space extensions $\wt A_\tau $ with $\tau\in\RCH$.
\begin{theorem}\label{th3.8}
Assume that $\Pi=\bt$ is a boundary triplet for $A^*$, $\tau\in\RCH$, $\wt A_\tau=\wt A_\tau^*$ is the corresponding exit space extension of $A$ and $\CAt$ is the compression of $\wt A_\tau$. Moreover, let $\tau_0\in R_c[\cH_0]$ and $\cK$ be the operator and multivalued parts of $\tau$ respectively (see \eqref{2.9}) and let  $\cB_{\tau_0}\in\B (\cH_0)$ and $\cN_{\tau_0}:\dom\cN_{\tau_0}\to \cH_0 \;(\dom\cN_{\tau_0}\subset\cH_0)$ be operators corresponding to $\tau_0$ in accordance with Proposition \ref{pr2.1.1}. Then $\CAt=A_{\t_c}$ (in the triplet $\Pi$) with the symmetric linear relation $\t_c$ in $\cH$ given by
\begin{gather}\label{3.13}
\t_c=-\cN_{\tau_0}\dotplus \wh\mul \t_c=\{\{h,-\cN_{\tau_0}h+h'\}: h\in\dom \cN_{\tau_0}, h'\in  \mul \t_c\}.
\end{gather}
Moreover,
\begin{gather}\label{3.13.1}
\ran\cB_{\tau_0}\oplus \cK\subset \mul \t_c, \qquad \ov{\mul \t_c}=\ov{\ran\cB_{\tau_0}} \oplus \cK.
\end{gather}
If in addition $\ran \cB_{\tau_0}$ is closed, then $\ov{\mul \t_c}=\mul \t_c = \ran\cB_{\tau_0} \oplus \cK$ and hence
\begin{gather}\label{3.14}
\t_c=\{\{h,-\cN_{\tau_0}h+\cB_{\tau_0}\psi+k \}: h\in\dom \cN_{\tau_0}, \psi\in  \cH_0, k\in\cK\}.
\end{gather}
\end{theorem}
\begin{proof}
Assume that $\tau=\{\cH'\oplus\cH''\oplus\cK,B_1,B_2,\tau_1\}$ with $\cH' \oplus \cH''=\cH_0$ and $\tau_1\in R_u[\cH']$ (see Corollary \ref{cor2.4} and Remark \ref{rem2.4a}). Moreover, let $\G_0$ and $\G_1$ have the block representations  \eqref{2.20} and let $\t_0\in \C (\cH)$ be given by \eqref{2.20.1}. Then according to Proposition \ref{pr3.3} $S=S(\wt A_\tau)$ admits the representation \eqref{2.21} and by Lemma \ref{lem2.7}, (ii) the equalities \eqref{2.23} define a boundary triplet $\Pi'=\{\cH',\G_0',\G_1'\}$ for $S^*$. Moreover, according to Lemma \ref{lem3.2} $\wt A_\tau=\wt S_{\tau_1}$ (in the triplet $\Pi'$) and by Proposition \ref{pr3.7} there exist a Hilbert space $\gH_r$, a symmetric relation $S_r\in\C (\gH_r)$ and a boundary triplet $\Pi_r=\{\cH',\G_0^r,\G_1^r\}$ for $S_r^*$ such that $\tau_1$ is the Weyl function of $\Pi_r$ and $\CAt=S_{-\cF_r}$ (in the triplet $\Pi'$), where $\cF_r$ is the forbidden relation of $\Pi_r$. It follows from Theorem \ref{th2.12} that
\begin{gather}
\ran \cB_{\tau_1}\subset \mul \cF_r, \qquad \ov{\mul \cF_r} =\ov {\ran \cB_{\tau_1}}\label{3.15}\\
-\cF_r=-\cN_{\tau_1}\dotplus \wh\mul \cF_r=\{\{h,-\cN_{\tau_1}h+h'\}:h\in \dom \cN_{\tau_1}, h'\in\mul\cF_r \},\label{3.16}
\end{gather}
where $\cN_{\tau_1}:\dom \cN_{\tau_1}\to \cH'$ is the operator given by
\begin{gather}
\dom \cN_{\tau_1}=\{h\in\cH': \lim_{y\to \infty}y\im (\tau_1(iy)h,h)<\infty\}\label{3.17}\\
\cN_{\tau_1}h= \lim_{y\to \infty}\tau_1(iy)h, \quad h\in  \dom \cN_{\tau_1}.\label{3.17.0}
\end{gather}
Let $\dom\t_c':=\dom \cN_{\tau_1}\oplus\cH''$ and $\t_c':\dom\t_c'\to \cH_0$ be the operator given by
\begin{gather}\label{3.17.1}
\t_c'=\begin{pmatrix} - \cN_{\tau_1} & B_1\cr B_1^*\up \dom \cN_{\tau_1} & B_2\end{pmatrix}:\dom \cN_{\tau_1}\oplus\cH''\to \cH'\oplus\cH''.
\end{gather}
Then by Lemma \ref{lem2.7} $\CAt=A_{\t_c}$ (in the triplet $\Pi$) with the linear relation $\t_c$ in $\cH$ given by
\begin{gather}
\mul\t_c=\mul\cF_r\oplus\cK\label{3.17.2}\\
\t_c=\t_c'\dotplus \wh\mul\t_c=\{\{h,\t_c'h+h'\}: h\in\dom\t_c',\, h'\in \mul\t_c \}\label{3.17.3}
\end{gather}
It follows from \eqref{2.8} (with $\tau_0(\l)$ in place of $\tau(\l)$) that
\begin{gather*}
y\im (\tau_0(iy)h,h)=y\im (\tau_1(iy) P_{\cH'}h, P_{\cH'}h), \quad h\in\cH_0.
\end{gather*}
Therefore by \eqref{3.17} $\dom \t_c'=\dom \cN_{\tau_0}$. Moreover, combining \eqref{3.17.1}, \eqref{3.17.0} and \eqref{2.9.1.1} for each $h=h'\oplus h''\in\dom\t_c'$ one obtains
\begin{gather*}
\t_c' h=(- \cN_{\tau_1} h' +B_1 h'')\oplus (B_1^* h'+ B_2 h'')=\\
-\lim_{y\to \infty}((\tau_1 (iy) h' -B_1 h'')\oplus (-B_1^* h'- B_2 h''))= -\lim_{y\to \infty}\tau_0 (iy) h=-\cN_{\tau_0} h.
\end{gather*}
This and  \eqref{3.17.3} yield \eqref{3.13}.

Next, by \eqref{2.9.1.1}  $\cB_{\tau_1}=\cB_{\tau_0}P_{\cH'}$ and hence $\ran \cB_{\tau_1}=\ran\cB_{\tau_0}$. Therefore by \eqref{3.15} and \eqref{3.17.2} the relations \eqref{3.13.1} are valid. Finally,  $\t_c$ is symmetric in view Proposition \ref{pr2.6}, (\romannumeral 5).
\end{proof}
\begin{theorem}\label{th3.10}
Assume that $\Pi=\bt$ is a boundary triplet for $A^*$,  $A_0=\ker \G_0$ (that is $A_0$ is a fixed canonical self-adjoint extension of $A$ in the Krein formula \eqref{2.31}), $\tau \in \wt R(\cH)$ and $\tau_0\in R [\cH_0]$ is the operator part of $\tau$ (see \eqref{2.9.3}). Then the following statements are equivalent:

{\rm (i)} $T(\wt A_\tau)$ is closed and $\CAt\subset A_0$;

{\rm (ii)} $\tau \in \RCH$ and $\tau_0$ satisfies
\begin{gather}\label{3.18}
\lim_{y\to\infty} y \im (\tau_0(iy)h,h)=\infty, \quad h\in\cH_0, \;\; h\neq 0.
\end{gather}

If statement {\rm (i)} (or equivalently {\rm (ii)}) is valid then
\begin{gather}\label{3.19}
\ov{\CAt}=\{\wh f\in A^*: \G_0\wh f=0, \, \G_1 \wh f\in \ov{\ran \cB_{\tau_0}}\oplus \cK\}.
\end{gather}
If in addition $\ran \cB_{\tau_0}$ is closed, then $\CAt$ is closed and
\begin{gather}\label{3.19.1}
\CAt=\{\wh f\in A^*: \G_0\wh f=0, \, \G_1 \wh f=
\cB_{\tau_0}h\oplus k \;\;\text{\rm with some}\;\; h\in\cH_0 \;\; {\rm and} \;\; k\in\cK\}.
\end{gather}
\end{theorem}
\begin{proof}
Assume statement (i). Then by \eqref{3.4.1} $\SAt\subset A_0$ and hence $\SAt\wh + A_0=A_0$. Therefore $\SAt\wh + A_0$ is closed and by Theorem \ref{th3.5}, (i) $\tau \in \RCH$. This in view of Theorem \ref{th3.8} implies that $\CAt= A_{\t_c}$ with  $\t_c$ given by \eqref{3.13}. It follows from \eqref{3.13} that $\CAt\subset A_0$ if and only if $\dom\cN_{\tau_0}=\{0\}$, which yields \eqref{3.18}. Thus (ii) holds.

Conversely, let (ii) holds. Then by Theorem \ref{th3.5}, (i) $T(\wt A_\tau)$ is closed. Moreover, by  Theorem \ref{th3.8} $\CAt=A_{\t_c}$ with  $\t_c$ of the form  \eqref{3.13} and \eqref{3.18} shows that $\dom\cN_{\tau_0}=\{0\}$. Hence $\t_c= \{0\}\oplus \mul\t_c$ and therefore $\CAt\subset A_0$, which yields statement (i). Moreover, by Proposition  \ref{pr2.6}, (\romannumeral 4) $\ov{\CAt}=A_{\ov\t_c}$, where in view of \eqref{3.13.1} $\ov\t_c=\{0\}\oplus \ov{\mul\t_c}=\{0\}\oplus \cH_c$ with $\cH_c=\ov{\ran\cB_{\tau_0}}\oplus \cK$. This yields \eqref{3.19}. Finally, the last statement of the theorem follows from the previous one.
 \end{proof}
\begin{corollary}\label{cor3.11}
Let the assumptions be the same as in Theorem \ref{th3.10}. Then the following statements are equivalent:

{\rm (i)} $\TAt$ is closed and $\ov  {\CAt}=A_0$;

{\rm (ii)} $\tau \in\RCH$ and $\ker \cB_{\tau_0}=\{0\}$.

If in addition $\ran \cB_{\tau_0}$ is closed, then statement {\rm (ii)} is equivalent to the following one:

{\rm(\romannumeral 1$\,^\prime$)} $\TAt$ is closed and $ {\CAt}=A_0$.
\end{corollary}
\begin{proof}
Assume that (i) holds. Then by Theorem \ref{th3.10} $\tau \in\RCH$ and $\ov {\CAt}=A_{\{0\}\oplus \cH_c}$ with $\cH_c=\ov{\ran \cB_{\tau_0}}\oplus\cK$. Since $A_0=A_{\{0\}\oplus\cH}$, it follows from $\ov  {\CAt}=A_0$ that $\cH_c=\cH=\cH_0\oplus\cK$. Hence $\ov{\ran \cB_{\tau_0}}=\cH_0$ or, equivalently, $\ker \cB_{\tau_0}=\{0\}$. Thus (ii) is valid.

Conversely, let (ii) holds. Then by Theorem \ref{th3.5}, (i) $\TAt$ is closed. Moreover, by Theorem \ref{th3.8} $\CAt=A_{\t_c}$ with $\t_c$ given by \eqref{3.13} and satisfying \eqref{3.13.1}. Since $\ov{\ran \cB_{\tau_0}}=\cH_0$, it follows from \eqref{3.13.1} that $\ov{\mul\t_c}=\cH$. Hence by Lemma \ref{lem2.4b}, (i) $\ov{\t_c}=\{0\}\oplus\cH$ and in view of Proposition \ref{pr2.6},  (\romannumeral 4) $\ov{\CAt}=A_{\{0\}\oplus\cH}=A_0$. This proves statement (i).

Assume now that $\ran \cB_{\tau_0}$ is closed. Then (i$\,^\prime$) implies (i) and, consequently, (ii). Conversely, let (ii) holds. Then (i) is valid and hence $\CAt=A_{\t_c}$ with $\t_c = \{0\}\oplus \mul\t_c$. Moreover, $\ran \cB_{\tau_0}=\cH_0$ and by Theorem \ref{th3.8} $\mul\t_c=\cH_0\oplus\cK=\cH$. Therefore $\CAt=A_{\{0\}\oplus\cH}=A_0$, which yields  (i$\,^\prime$).
\end{proof}
\begin{corollary}\label{cor3.12}
Assume that $\Pi=\bt$ is  a boundary triplet for $A^*$ and $\tau\in\wt R(\cH)$. Then $\TAt$ is closed and $\CAt = A$ if and only if $\tau\in R_c[\cH], \; \cB_\tau =0$ and
\begin{gather}\label{3.20}
\lim_{y\to\infty} y \im (\tau(iy)h,h)=\infty, \quad h\in\cH, \;\; h\neq 0.
\end{gather}
\end{corollary}
\begin{proof}
Let $\TAt$ be closed and $\CAt = A$. Then by Theorem \ref{th3.10} $\tau\in\RCH$, the operator-function $\tau_0\in R_c[\cH_0]$ in decomposition \eqref{2.9} of $\tau$ satisfies \eqref{3.18}  and $\ov{\ran \cB_{\tau_0}}\oplus\cK=\{0\}$. Hence $\cK=\{0\}$ and, consequently, $\cH_0=\cH$ and $\tau=\tau_0$. This implies that $\tau\in R_c[\cH]$  and \eqref{3.20} holds. Moreover, $\ran \cB_{\tau_0}=\{0\}$ and therefore $\cB_\tau=\cB_{\tau_0}=0$.

Conversely, let $\tau\in R_c[\cH],\; \cB_\tau =0$ and \eqref{3.20} is satisfied. Then by Theorem \ref{th3.10} $\TAt$ is closed and $\ov{\CAt}$ is given by \eqref{3.19} with $\ran\cB_{\tau_0}=\ran\cB_{\tau}=\{0\}$ and $\cK=\{0\}$. Hence $\CAt=\ov{\CAt}=\ker \G$ and by Proposition \ref{pr2.6}, (i) $\CAt=A$.
\end{proof}
\begin{remark}\label{rem3.13}
Assume that $A$ is a closed densely defined symmetric operator in $\gH$. Then each exit space extension $\wt A=\wt A^*$ of $A$ is a densely defined operator and according to M.A.~Naimark (see e.g. \cite{AG}) an extension $\wt A$ of $A$ is said to be of the second kind if $\dom \wt A \cap\gH =\dom A$ or equivalently if $C (\wt A)=A$. Clearly, in the case $\ov{\dom A}=\gH$ Corollary \ref{cor3.12} gives a parametrisation of  all extensions $\wt A$ of the second kind satisfying $\ov{T(\wt A)}=T(\wt A)$. Note that this result follows from the results of \cite{DM09}, where a parametrisation of all extensions $\wt A \supset A$ of the second kind of a densely defined $A$ was obtained in terms of the parameter $\tau$ from the Krein formula \eqref{2.31}. Observe also that a somewhat  other parametrization of the second kind extensions can be found  in \cite{Sht66,Zag13}.
\end{remark}
\subsection{Exit space self-adjoint extensions with a self-adjoint compression}
In the following proposition we provide a sufficient condition on the parameter $\tau$ for self-adjointness of the compression $\CAt$.
\begin{proposition}\label{pr3.14}
Assume that $\Pi=\bt$ is a boundary triplet for $A^*$, $\tau \in\RCH$ and $\tau_0\in R_c[\cH_0]$ is the operator part of $\tau$ (see \eqref{2.9}). If  $\ran\cB_{\tau_0}$ is closed and
\begin{gather}\label{3.20.1}
\lim_{y\to\infty} y\im (\tau_0(iy)h,h)<\infty, \quad h\in \ker\cB_{\tau_0},
\end{gather}
then $\CAt$ is self-adjoint.
\end{proposition}
\begin{proof}
Assume that $\ran\cB_{\tau_0}$ is closed and \eqref{3.20.1} holds. Then according to Theorem \ref{th3.8} $\CAt=A_{\t_c}$ with  a symmetric linear relation $\t_c$ in $\cH$ given by \eqref{3.13} and satisfying $\ov{\mul \t_c}=\mul \t_c= \ran \cB_{\tau_0}\oplus\cK$. Moreover, by \eqref{3.20.1} $\ker \cB_{\tau_0}\subset \dom \cN_{\tau_0}=\dom\t_c$. Therefore
\begin{gather*}
\cH\ominus\mul \t_c=\cH_0\ominus \ran \cB_{\tau_0}=\ker \cB_{\tau_0}\subset \dom\t_c
\end{gather*}
and by Lemma \ref{lem2.4b}, (ii) $\t_c=\t_c^*$. Now the required statement follows from Proposition \ref{pr2.6}, ~(\romannumeral 5).
\end{proof}
\begin{corollary}\label{cor3.15}
Let the assumptions be the same as in Proposition \ref{pr3.14}. If $\ran\cB_{\tau_0}$ is closed and there is a compact interval $[a,b]\subset \bR$ such that $\tau_0$ admits a holomorphic continuation onto $\bR\setminus [a,b]$, then $\CAt$ is self-adjoint.
\end{corollary}
\begin{proof}
If the function $\tau_0$ is holomorphic in $\bC\setminus [a,b]$, then it admits the Laurent expansion
\begin{gather*}
\tau_0(\l)=C_0 +\l\cB_{\tau_0}+\sum_{k=1}^{\infty}\frac 1 {\l^k}C_k,\quad |\l|> R> 0
\end{gather*}
with operator-valued coefficients  $C_k=C_k^*\in\B (\cH_0), \;k=0,1,2, \dots$. Therefore
\begin{gather*}
\im (\tau_0(iy)h,h)=\sum_{m=1}^{\infty} \frac{(-1)^m}{y^{2m-1}} (C_{2m-1}h,h), \quad y\in\bR,\;\; h\in \ker \cB_{\tau_0}
\end{gather*}
and,consequently,
\begin{gather*}
\lim_{y\to\infty}y\im (\tau_0(iy)h,h)=-(C_1h,h)<\infty, \quad  h\in \ker \cB_{\tau_0}.
\end{gather*}
Thus \eqref{3.20.1} holds and the required statement follows from  Proposition \ref{pr3.14}.
\end{proof}
In the following theorem we describe exit space self-adjoint extensions $\wt A$ of $A$ such that the compression $C(\wt A)$  of $\wt A$ is self-adjoint and transversal with  $A_0$.
\begin{theorem}\label{th3.16}
Assume the conditions of Theorem \ref{th3.10}. Then the following statements are equivalent:

{\rm (i)} $\TAt$ is closed, $\CAt^*=\CAt$ and the extensions $\CAt$ and $A_0$ are transversal;

{\rm (ii)} $\tau\in R_c[\cH]$ and
\begin{gather}\label{3.21}
\lim_{y\to\infty} y\im (\tau(iy)h,h)<\infty, \quad h\in \cH.
\end{gather}
Moreover, if {\rm (i)}  (equivalently {\rm (ii)}) is satisfied, then
\begin{gather}\label{3.22}
\CAt=A_{-\cN_\tau}=\{\wh f\in A^*: \G_1\wh f = -\cN_\tau \G_0\wh f \},
\end{gather}
where $\cN_\tau\in\B(\cH)$ is the operator given by
\begin{gather}\label{3.22.1}
\cN_\tau=s \text{-} \lim_{y\to \infty} \tau(iy).
\end{gather}
\end{theorem}
\begin{proof}
Let statement (i) holds. Since $\SAt\subset \CAt$, it follows from Lemma \ref{lem2.7.1} that the linear  relation $\SAt\wh + A_0$ is closed. Moreover, by \eqref{3.4.1} $\SAt\cap A_0=A$. Therefore by Theorem \ref{th3.5}, (ii) $\tau\in R_c[\cH]$. Next, according to Theorem \ref{th3.8} $\CAt=A_{\t_c}$, where $\t_c$ is given by \eqref{3.13} with $\tau_0=\tau$, and by Proposition \ref{pr2.6}, (\romannumeral 7) $\t_c\in\B (\cH)$. Therefore by \eqref{3.13} $\dom\cN_\tau=\dom\t_c=\cH$, which in view of \eqref{2.3.1} yields \eqref{3.21}. Thus statement (ii) is valid. Moreover, since $\mul\t_c=\{0\}$, it follows from \eqref{3.13} that $\t_c=-\cN_\tau$ and in view of definition    \eqref{2.3.2} of $\cN_\tau$ the equality \eqref{3.22.1} holds. This yields \eqref{3.22}.

Conversely, assume (ii). Then by Theorem \ref{th3.5} $\TAt$ is closed. Moreover, by Theorem \ref{th3.8} $\CAt=A_{\t_c}$, where $\t_c$ is given by \eqref{3.13} with $\tau_0=\tau$ and \eqref{3.21} yields $\dom\t_c=\dom\cN_\tau=\cH$. Since $\t_c$ is symmetric, this implies that $\t_c=\t_c^*\in \B(\cH)$. Therefore by Proposition \ref{pr2.6}, (\romannumeral 5) and (\romannumeral 7) $\CAt^*=\CAt$ and the extensions $\CAt$ and $A_0$ are transversal.
\end{proof}
\subsection{Finite-codimensional exit space extensions}
Let $\wt\gH\supset\gH$ be a Hilbert space, let $\gH_r=\wt\gH\ominus \gH$ and let $\wt A=\wt A^*\in \C (\wt\gH)$ be an  exit space extension of $A$. Such an  extension is called finite-codimensional, if $n_r:=\dim \gH_r<\infty$.

The following proposition is well known (see e.g. \cite{DajLan18}).
\begin{proposition}\label{pr3.17}
Let $A\in \C (\gH)$ be a symmetric linear relation  with finite deficiency indices $n_+(A)=n_-(A)=:d<\infty$ and let $\Pi=\bt$ be a boundary triplet for $A^*$ (hence $\dim \cH=d<\infty$). Moreover, let $\tau \in\RH$ and let $\tau_0\in R [\cH_0]$ be the operator part of $\tau$ (see \eqref{2.9}). Then $\wt A_\tau$ is finite-codimensional if and only if
\begin{gather}\label{3.23.1}
\tau_0(\l)=\cA+\l\cB + \sum_{j=1}^l \frac 1 {\a_j-\l}\cA_j, \quad \l\in\CR
\end{gather}
where
\begin{gather}\label{3.23.2}
\a_j\in\bR, \;\;\cA,\cB,\cA_j\in\B(\cH_0), \;\; \cA=\cA^*, \;\cB\geq 0, \; \cA_j\geq 0,\; \cA_j\neq 0,\;\; j=1,2,\dots, l.
\end{gather}
Moreover,
\begin{gather}\label{3.24}
n_r=\dim\ran \cB + \sum_{j=1}^l  \dim\ran \cA_j.
\end{gather}
\end{proposition}
In the following two theorems we extend Proposition \ref{pr3.17} to a ceratin class of self-adjoint extensions $\wt A$ of a symmetric relation $A$ with possibly infinite deficiency indices $n_+(A)=n_-(A)\leq \infty$ and characterise compressions $C(\wt A)$ of such extensions.
\begin{theorem}\label{th3.18}
Let $A\in\C (\gH)$ be a symmetric linear relation in $\gH$ with equal deficiency indices $n_+(A)=n_-(A)\leq\infty$, let $\Pi=\bt$ be a boundary triplet for $A^*$, let $A_0=\ker\G_0$, let $\tau \in \RH$ and let $\wt A_\tau=\wt A_\tau^*\in\C (\wt\gH)\;(\wt\gH\supset \gH)$ be the corresponding exit space extension of $A$ (see Theorem \ref{th2.14}). Assume also that $\tau_0\in R[\cH_0]$ is the operator part of $\tau$ (see \eqref{2.9}). Then the following statements are equivalent:

{\rm (i)} $\wt A_\tau$ is finite-codimensional and the linear relation $\SAt\wh + A_0$ is closed;

{\rm (ii)} $\tau_0$ is the rational function \eqref{3.23.1}, \eqref{3.23.2} such that $\dim\ran \cB<\infty$ and $\dim \ran A_j<\infty,\; j=1,2,\dots, l$.

Moreover, if statement {\rm (ii)} (and hence {\rm (i)}) is valid, then the dimension $n_r$ of $\gH_r=\wt\gH\ominus \gH$ satisfies \eqref{3.24}.
\end{theorem}
\begin{proof}
Assume statement (i) and let $S_r= S_r(\wt A_\tau)$ and $T_r(\wt A_\tau)$ be linear relations in $\gH_r$ given by
\begin{gather*}
S_r=\wt A_\tau\cap \gH_r^2, \qquad T_r(\wt A_\tau) =\{\{P_{\gH_r} f, P_{\gH_r} f' \}:\{f,f'\}\in \wt A_\tau\}.
\end{gather*}
Since dim $\gH_r<\infty$, it follows that $ T_r(\wt A_\tau)$ is closed and according to \cite[Proposition 2.15]{DM06} $T(\wt A_\tau)$ is closed as well. Therefore  by  Theorem \ref{th3.5} $\tau\in\RCH$. Let in accordance with Corollary \ref{cor2.4} $\tau_0$ be decomposed as in \eqref{2.9.1.1} with $\tau_1\in R_u[\cH']$, that is $\tau=\{\cH'\oplus\cH''\oplus\cK,B_1,B_2,\tau_1\}$ in the sense of Remark \ref{rem2.4a}. Then by Proposition \ref{pr3.3} the relation $S=\SAt$ admits the representation \eqref{2.21}. Let $\Pi'=\{\cH',\G_0',\G_1'\}$ be boundary triplet \eqref{2.23} for $S^*$. Then by Lemma \ref{lem3.2} $\wt A_\tau= \wt S_{\tau_1}$ (in the triplet $\Pi'$). Moreover, according to \cite[Lemma 2.14]{DM06} $n_\pm(S)=n_\pm(S_r)$ and, consequently, $n_+(S)=n_-(S)<\infty$. Therefore $\dim\cH'<\infty $ and by Proposition \ref{pr3.17}
\begin{gather}\label{3.27}
\tau_1(\l)=\cA_0+\l\cB_0 + \sum_{j=1}^l \frac 1 {\a_j-\l}\cA_{0j}, \quad \l\in\CR,
\end{gather}
where $\a_j\in\bR$ and $\cA_0,\cB_0,\cA_{0j}\in\B(\cH')$ are the operators satisfying $\cA_0=\cA_0^*$, $\cB_0\geq 0$ and $\cA_{0j}\geq 0, \cA_{0j}\neq 0, \; j=1,2,\dots, l$.This and the block representation \eqref{2.9.1.1} of $\tau_0(\l)$ imply that $\tau_0(\l)$  is of the form \eqref{3.23.1} with
\begin{gather}\label{3.27.1}
\cA=\begin{pmatrix} \cA_0 & - B_1 \cr -B_1^* & -B_2\end{pmatrix}, \qquad \cB=\begin{pmatrix} \cB_0& 0\cr 0 &0 \end{pmatrix}, \qquad \cA_j=\begin{pmatrix} \cA_{0j}& 0\cr 0 &0 \end{pmatrix}, \;\; j=1,2, \dots, l.
\end{gather}
Hence $\cA=\cA^*, \; \cB\geq 0, \; \cA_j\geq 0, \;\cA_j\neq 0$ and
\begin{gather}\label{3.27.2}
\dim\ran \cB=\dim\ran \cB_0 <\infty, \qquad \dim\ran\cA_j=\dim\ran\cA_{0j}<\infty,
\end{gather}
which yields statement (ii).

Conversely, let statement (ii) holds. Put $\f_0(\l)=\im\l$ and $\f_j(\l)=\im \frac1 {\a_j-\l}$. Then by \eqref{3.23.1}
\begin{gather}\label{3.28}
\im\tau_0(\l)=\f_0(\l)\cB + \sum_{j=1}^l \f_j(\l)\cA_j, \quad \l\in \CR.
\end{gather}
Since $\im\l\cd \im\f_j(\l)>0,\; j=0,1, \,\dots,\, l, $ it follows from \eqref{3.28} that
\begin{gather}\label{3.29}
\cH'':=\ker\im \tau_0(\l)=\ker\cB\cap\left(\bigcap_{j=1}^l \ker \cA_j \right)
\end{gather}
and, consequently, $\cH':=\cH_0\ominus\cH''=\ran \cB+\ran A_1+\dots +\ran A_l$. Therefore $\dim\cH'<\infty  $ and in view of \eqref{3.29} block representations \eqref{3.27.1} hold (with respect to  the decomposition $\cH_0=\cH'\oplus \cH'')$. Combining \eqref{3.27.1} with \eqref{3.23.1} one obtains the block representation \eqref{2.9.1.1} of $\tau_0(\l)$ with $\tau_1\in R[\cH']$ given by \eqref{3.27}. Since $\dim\cH'<\infty$, it follows that $\tau_1\in R_u[\cH']$ and according to  Proposition \ref{pr2.2} $\tau_0\in R_c [\cH_0]$. Therefore $\tau \in\RCH$ and by Theorem \ref{th3.5} $\SAt\wh + A_0$ is closed.

Now it remains to show that $\wt\cA_\tau$ is finite-codimensional and \eqref{3.24} holds. Since $\tau=\{\cH'\oplus\cH''\oplus \cK, B_1,B_2,\tau_1\}$, it follows from Proposition \ref{pr3.3} that $S=\SAt$ admits the representation \eqref{2.21} and by Lemma \ref{lem2.7} the equalities \eqref{2.23} define a boundary triplet $\Pi'=\{\cH',\G_0', \G_1'\}$ for $S^*$. Therefore by Proposition \ref{pr2.5} $n_\pm(S)=\dim\cH'<\infty$. Moreover, by Lemma \ref{lem3.2} $\wt A_\tau =\wt S_{\tau_1}$ (in the triplet $\Pi'$). Thus in accordance with Proposition \ref{pr3.17} $\wt A_\tau$ is finite-codimensional and
\begin{gather*}
n_r=\dim\ran \cB_0 + \sum_{j=1}^l  \dim\ran \cA_{0j}.
\end{gather*}
Combining this equality with \eqref{3.27.2} we arrive at \eqref{3.24}.
\end{proof}
\begin{theorem}\label{th3.19}
Let the assumptions of Theorem \ref{th3.18} be satisfied and let claim {\rm (ii)}  (equivalently {\rm (i)}) of this theorem be valid. Moreover, let $\cA'=P_{\ker\cB}\cA  \up \ker \cB\in \B (\ker\cB)$, let $\cK$ be the multivalued part of $\tau$ (see \eqref{2.9}) and let $\t_c$ be a self-adjoint linear relation in $\cH$ given by
\begin{gather}\label{3.35}
\t_c=\{\{h,-\cA' h\oplus h'\oplus k\}:h\in \ker\cB, \, h'\in\ran\cB, k\in\cK\}
\end{gather}
(this means that $-\cA'$ is the operator part of $\t_c$ and $\mul\t_c=\ran\cB\oplus\cK$). Then the extension $\wt A_\tau$ is finite-codimensional, the compression $\CAt$ of  $\wt A_\tau$ is  self-adjoint and $\CAt=A_{\t_c}$ (in the triplet $\Pi$).

Moreover, the following holds:

{\rm (i)} $\CAt=A_0$ if and only if $\ker \cB=0$ (that is $\cB>0$);

{\rm (ii)} $\CAt$ and $A_0$ are transversal if and only if $\tau\in R[\cH]$ and $\cB=0$, in which case $\t_c=-\cA$.
\end{theorem}
\begin{proof}
It follows from Theorem \ref{th3.18} that $\wt A_\tau$ is finite-codimensional. Moreover, by Theorem \ref{th3.8}  $\CAt=\t_c$ with symmetric relation $\t_c$ in $\cH$ given by \eqref{3.13}. Since in view of \eqref{3.23.1}
\begin{gather*}
y \im (\tau_0(iy)h,h)=y^2 (\cB h,h)+\sum_{j=1}^l \frac {y^2}{\a_j^2+y^2}(\cA_j h,h), \quad h\in\cH_0,
\end{gather*}
it follows that $\lim\limits_{y\to \infty}y\im (\tau_0(iy)h,h)<\infty$ if and only if $h\in\ker\cB$. Therefore by definitions \eqref{2.3.1} and \eqref{2.3.2} $\dom\cN_{\tau_0}=\ker\cB$ and
\begin{gather*}
\cN_{\tau_0}h=\lim_{y\to\infty}\tau_0(iy) h =\cA h, \quad h\in\ker\cB(= \dom\cN_{\tau_0})
\end{gather*}
Moreover, by \eqref{2.3} $\cB_{\tau_0}=\cB$ and the finite-dimensionality of $\ran\cB$ yields $\ran \cB_{\tau_0}= \ov{\ran \cB_{\tau_0}}=\ran\cB$. Therefore by Theorem \ref{th3.8} $\mul \t_c=\ran\cB\oplus\cK$ and \eqref{3.13} can be written as
\begin{gather}\label{3.36}
\t_c=\{\{h,-\cA h+( h'\oplus k)\}:h\in \ker\cB, \, h'\in\ran\cB, k\in\cK\}
\end{gather}
Since $\cH_0=\ker\cB\oplus\ran\cB$ and $h'\in \ran\cB$ in \eqref{3.36} is arbitrary, the equality \eqref{3.36} is equivalent to \eqref{3.35}.  Finally, statements (i) and (ii) are implied by Corollary \ref{cor3.11} and Theorem \ref{th3.16}.
\end{proof}
\subsection{The case of finite deficiency indices}
In this subsection we suppose that the following assumption (a) is satisfied:

(a) $A$ is a closed symmetric linear relation in $\gH$ with finite deficiency indices $n_+(A)=n_-(A)<\infty$, $\Pi=\bt$ is a boundary triplet for $A^*$, $\tau\in\RH$, $\tau_0\in R[\cH_0]$ is the operator part of $\tau$ (see \eqref{2.9}), $\wt A_\tau=\wt A_\tau^*$ is the corresponding exit space extension of $A$ and $\CAt$ is the compression of $\wt A_\tau$.

Then $\ex=\cex$ and hence for each exit space extension $\wt A=\wt A^*$ of $A$ the linear relations $C(\wt A), \; T(\wt A)$ and $S(\wt A)\wh + A_0$ are closed (here $A_0=\ker \G_0$). Moreover, under the assumption (a) $\dim\cH<\infty$, which implies that $\RH=\RCH$, $R [\cH']=R_c [\cH']$ and  $\ov{\ran\cB_\tau}= \ran\cB_\tau$ for any subspace $\cH'\subset \cH$ and operator-function $\tau\in R[\cH']$. These facts together with  Theorems \ref{th3.8}, \ref{th3.10} and Corollaries \ref{cor3.11}, \ref{cor3.12} yield the following two theorems.
\begin{theorem}\label{th3.22}
Let the assumption (a) be satisfied. Then $\CAt=A_{\t_c}$ (in the triplet $\Pi$) with the symmetric linear relation $\t_c$ in $\cH$ defined by \eqref{3.14}.
\end{theorem}
\begin{theorem}\label{th3.23}
Let the assumption (a) be satisfied and let $A_0=\ker\G_0$. Then:

{\rm (i)} $\CAt\subset A_0$ if and only if the operator function $\tau_0$ satisfies
\begin{gather}\label{3.38}
\lim_{y\to\infty} y \im (\tau_0(iy)h,h)=\infty, \quad h\in\cH_0, \;\; h\neq 0.
\end{gather}
In this case $\CAt$ is defined by abstract boundary conditions \eqref{3.19.1}.

{\rm (ii)} $\CAt=A_0$ if and only if $\ker\cB_{\tau_0}=0$ (that is $\cB_{\tau_0}>0$ ).

{\rm (iii)} $\CAt=A$ if and only if $\tau\in R [\cH]$, $\cB_\tau=0$ and
\begin{gather}\label{3.39}
\lim_{y\to\infty} y \im (\tau(iy)h,h)=\infty, \quad h\in\cH, \;\; h\neq 0.
\end{gather}
\end{theorem}
In the following theorem we describe all exit space  extensions $\wt A=\wt A^*$ of a symmetric relation  $A$ with finite deficiency indices such that the compression $C(\wt A)$  of $\wt A$ is self-adjoint.
\begin{theorem}\label{th3.24}
Let the assumption (a) be satisfied and let $A_0=\ker\G_0$. Then:

{\rm (i)} $\CAt$ is self-adjoint if and only if
\begin{gather}\label{3.40}
\lim_{y\to\infty} y \im (\tau_0(iy)h,h)<\infty, \quad h\in\ker \cB_{\tau_0}.
\end{gather}
If in particular $\tau_0$ is holomorphic out of a compact interval $[a,b]\subset \bR$, then $\CAt$ is self-adjoint.

{\rm (ii)}  $\CAt$ is self-adjoint and transversal with $A_0$ if and only if $\tau\in R [\cH]$ and
\begin{gather*}
\lim_{y\to\infty} y \im (\tau(iy)h,h)<\infty, \quad h\in\cH.
\end{gather*}
In this case
\begin{gather*}
\CAt=A_{-\cN_\tau}=\{\wh f\in A^*: \G_1 \wh f =-\cN_\tau \G_0 \wh f \},
\end{gather*}
where $\cN_\tau=\cN_\tau^*\in\B (\cH)$ is the operator given by $\cN_\tau=\lim\limits_{y\to\infty} \tau (iy)$.
\end{theorem}
\begin{proof}
(i) According to Theorem \ref{th3.22} $\CAt=A_{\t_c}$ with  a symmetric linear relation $\t_c$ in $\cH$ defined by \eqref{3.14}. It follows from \eqref{3.14} that $\mul\t_c=\ran\cB_{\tau_0}\oplus\cK$ and therefore $\cH\ominus \mul\t_c=\cH_0\ominus \ran\cB_{\tau_0}=\ker \cB_{\tau_0}$. Moreover by \eqref{3.14} $\dom\t_c=\dom\cN_{\tau_0}$ and hence the inclusion $\cH\ominus\mul\t_c \subset \dom\t_c$ is equivalent to \eqref{3.40}. Therefore by Lemma \ref{lem2.4b}, (iii) \eqref{3.40} is equivalent to self-adjointness of $\t_c$, which by Proposition \ref{pr2.6}, (\romannumeral 5) is equivalent to $\CAt^*=\CAt$. Thus $\CAt^*=\CAt$ if and only if \eqref{3.40} holds. The second part of the statement (i) is implied by Corollary \ref{cor3.15}.

(ii) Statement (ii) directly follows from Theorem \ref{th3.16} and arguments before Theorem \ref{th3.22}.
\end{proof}
In the following theorem we characterize compressions of finite-codimensional extensions $\wt A=\wt A^*$ of a symmetric relation $A$ with finite deficiency indices. The statements of this theorem follow from Theorem \ref{th3.19} and arguments before Theorem \ref{th3.22}.
\begin{theorem}\label{th3.25}
Let under the assumption (a) $\tau_0$ be a rational function \eqref{3.23.1}, \eqref{3.23.2} (according to Proposition \ref{pr3.17} this assumption is equivalent to the finite-codimensionality of $\wt A_\tau$). Moreover, let $\cA'=P_{\ker \cB}\cA\up \ker\cB$ and let $\cK$ be the multivalued part of $\tau$ (see \eqref{2.9}). Then $\CAt$ is a self-adjoint linear relation in $\gH$ and $\CAt=A_{\t_c}$ with the relation $\t_c=\t_c^*\in\C (\cH)$ given by \eqref{3.35}. Moreover, statements {\rm (i)} and {\rm (ii)} of Theorem \eqref{th3.19} are valid.
\end{theorem}
\begin{corollary}\label{cor3.26}
Let $\wt\gH$ be a Hilbert space, let $\gH$ be a subspace of a finite codimension in $\wt \gH$ (this means that $\dim (\wt\gH\ominus \gH)<\infty$)
and let $\wt A=\wt A^*$ be a linear relation in $\wt\gH$. Then the compression
$C(\wt A)=P_\gH \wt A\up\gH$  of
$\wt A$ (see \eqref{3.4}) is a self-adjoint linear relation in $\gH$.
\end{corollary}
\begin{proof}
Clearly, $\wt A$ is an exit space extension of the symmetric relation $S=S(\wt A)$ (see \eqref{3.3}). Moreover, it was shown in the proof of Theorem \ref{th3.18} that $n_+(S)=n_-(S)< \infty$. Therefore by Theorem \ref{th3.25}  $C(\wt A)$ is self-adjoint.
\end{proof}
\begin{remark}\label{rem3.27}
Statement of Corollary \ref{cor3.26} was proved by another method in \cite{ADW13} (for the case of an operator $\wt A=\wt A^*$ see \cite{Ste68}).
\end{remark}


\begin{thebibliography}{DHS}
\bibitem{AG}
N. I. Akhiezer, I. M.Glazman, \textit{Theory of linear operators in Hilbert space}, vol. I and II, Pitman, Boston-London-Melbourne,
1981.

\bibitem{AziDaj12}
T.Ya. Azizov, A. Dijksma, \textit{Closedness and adjoints of products of operators, and compressions}, Int. Equ. Oper. Theory \textbf{74} (2012), 259–-269.

\bibitem{ADW13}
T.Ya. Azizov, A. Dijksma, and G. Wanjala, \textit{Compressions of maximal dissipative and self-adjoint linear relations and of dilations}, Linear Algeb. Appl. \textbf{439} (2013), 771–-
792.

\bibitem{ACD16}
T.Ya. Azizov, B. \'Curgus, and  A. Dijksma, \textit{ Finite-codimensional compressions of
symmetric and self-adjoint linear relations in Krein spaces}, Int. Equ. Oper. Theory \textbf{86} (2016), 71–-95.

%\bibitem{Bro78}
%M.S. Brodskii, \textit{Unitary operator colligations and their chracteristic functions}, Russian Mathematical Surveys \textbf{33} (1978), no. 4, 159–-191.	

\bibitem{Bru77}
V.~M.~Bruk, \textit{Extensions of symmetric relations}, Math. Notes \textbf{22} (1977), no.6, 953--958.

\bibitem{DM00}
V.A. Derkach, S. Hassi, M.M. Malamud, and H.S.V. de Snoo,\textit {
Generalized resolvents of symmetric operators and admissibility},
Methods of Functional Analysis and Topology \textbf{6} (2000),
no.~3 , 24--55.

\bibitem{DM06}
V. A. Derkach, S. Hassi, M. Malamud, and H. S.V. de Snoo,\textit{Boundary Relations and their Weyl families}, Trans. Amer. Math. Soc. \textbf{358} (2006), no.~12, 5351--5400.

\bibitem{DM09}
V.A. Derkach, S. Hassi, M.M. Malamud, and H.S.V. de Snoo, \textit
{Boundary relations and generalized resolvents of symmetric
operators}, Russian J. Math. Ph. \textbf{16} (2009), no.~1,
17--60.

\bibitem{DM91}
V.A.~Derkach and M.M.~Malamud, \textit {Generalized resolvents and
the boundary value problems for Hermitian operators with gaps}, J.
Funct. Anal. \textbf{95} (1991),1--95.

\bibitem {DM95}
V.A. ~Derkach and M.M. ~Malamud, \textit {The extension theory of
Hermitian operators and the moment problem}, J. Math. Sciences
\textbf{73} (1995), no.~2, 141-242.


%\bibitem{DM17}
%V. A. Derkach and M. M. Malamud, \textit {Extension theory of symmetric operators and boundary value problems}, in Proceedings of Institute of
%Mathematics  NAS of Ukraine, V.104, Institute of
%Mathematics  NAS of Ukraine, Kyiv, 2017, 573 p.

\bibitem{DajLan17}
A. Dijksma and H. Langer, \textit{Finite-dimensional self-adjoint extensions of a symmetric operator with finite defect and their compressions},  In: Advances in complex analysis and operator theory, Festschrift in honor of Daniel Alpay, Birkh\"auser, Basel, (2017), 135–-163.

\bibitem{DajLan18}
A. Dijksma and H. Langer,\textit{Compressions of self-adjoint extensions of a symmetric operator and M.G. Krein’s resolvent formula}, Integr. Equ. Oper. Theory \textbf{90:41} (2018).

%\bibitem{DLS93}
%A. Dijksma, H. Langer,  H.S.V. de Snoo, \textit{Eigenvalues and pole functions
%of Hamiltonian systems with eigenvalue depending boundary conditions}, Math.
%Nachr.  \textbf{161} (1993), 107--153.


\bibitem{DajSno74}
A. Dijksma and  H.S.V. de Snoo, \textit{Self-adjoint extensions of symmetric subspaces}, Pasif. J. Math. \textbf{54} (1974), no. 1, 71--100.

%\bibitem{GK}
%I. Gohberg, M.G. Krein, \textit{Theory and applications of Volterra operators
%in Hilbert space}, Transl. Math. Monographs, 24, Amer. Math. Soc., Providence,
%R.I., 1970

\bibitem{GorGor}
V.I.~Gorbachuk and M.L.~Gorbachuk, \textit{Boundary problems for
differential-operator equations}, Kluver Acad. Publ.,
Dordrecht-Boston-London, 1991. (Russian edition: Naukova Dumka,
Kiev, 1984).

\bibitem {KreLan71}
M.G.~Krein and H.~Langer,\textit{On defect subspaces and
generalized resolvents of a Hermitian operator in the space $\Pi_\kappa$}, Funct. Anal. Appl. \textbf{5} (1971/1972), 136--146,
217--228

\bibitem{KreLan73}
M.G Krein and H. Langer, \textit{\" Uber die $Q$-functions eines
$\pi$-hermiteschen operators in raume $\Pi_\kappa$}, Acta. Sci.
Math. (Szeged) \textbf{34} (1973), 191--230.


\bibitem{LanTex77}
H.~Langer and B.~Textorious, \textit{On generalized resolvents and
$Q$-functions of symmetric linear relations (subspaces) in Hilbert
space}, Pacif. J. Math.  \textbf{72}(1977), no.~1 , 135--165.



\bibitem {Mal92}
 M. M. ~Malamud, \textit{On the formula of generalized resolvents of a nondensely defined Hermitian operator}, Ukr. Math. Zh.
\textbf{44}(1992), no.~ 12, 1658--1688.


\bibitem{NaFo}
B. Sz.-Nagy, C. Foias, \textit{Harmonic Analysis of Operators in Hilbert Space}, Paris and Akad..Kiado,
Budapest, 1967.


\bibitem{Nud11}
M.A. Nudelman, \textit{ A generalization of Stenger’s lemma to maximal dissipative
operators}. Integr. Equ. Oper. Theory  \textbf{70} (2011), 301-–305.



\bibitem{Sht62}
A.V.Shtraus, \textit{On selfadjoint operators in the orthogonal sum of Hilbert spaces},  Dokl. Akad.
Nauk SSSR,  \textbf{144} (1962), no.3, 512–-515.

\bibitem{Sht66}
A. V. Shtraus, \textit{On one-parameter families of extensions of a symmetric operator}, Izv. Akad. Nauk SSSR, Ser. Mat. \textbf{30} (1966), 1325-1352.


\bibitem{Sht70}
A. V. Shtraus, \textit{Extensions and generalized resolvents of a symmetric operator which is not densely defined}, Izv. Akad. Nauk SSSR. Ser. Mat. \textbf{34} (1970),no.1, 175-202. (Russian); English translation: Mathematics of the USSR-Izvestiya, \textbf{4} (1970), no. 1, 179--208.


\bibitem{Ste68}
W. Stenger, \textit{On the projection of a selfadjoint operator}, Bull. Am. Math. Soc. \textbf{74}(1968), 369–-372.


\bibitem{Zag13}
S.M. Zagorodnyuk, \textit{Generalized resolvents of symmetric and isometric operators: the Shtraus approach}, Ann. Funct. Anal. \textbf{4} (2013), no. 1, 175--285.


\end{thebibliography}
\end{document}